\documentclass[11pt]{article}

\usepackage{amsmath}
\usepackage{amsfonts}
\usepackage{graphicx}
\usepackage{setspace}
\usepackage{amsmath}
\usepackage{amssymb}
\usepackage{latexsym}
\usepackage{amsmath, amsfonts,amssymb, amsthm, euscript,makeidx,color,mathrsfs}

\usepackage[numbers,sort&compress]{natbib}

\def\sqr#1#2{{\vcenter{\vbox{\hrule height.#2pt
  \hbox{\vrule width.#2pt height#1pt \kern#1pt \vrule width.#2pt}
  \hrule height.#2pt}}}}
\usepackage{hyperref}
\hypersetup{
colorlinks=true,
linkcolor=blue,
urlcolor=cyan,
citecolor=cyan
}
\usepackage{caption}

\RequirePackage[capitalize,nameinlink]{cleveref}

\crefname{section}{section}{sections}
\crefname{subsection}{subsection}{subsections}
\Crefname{section}{Section}{Sections}
\Crefname{subsection}{Subsection}{Subsections}

\crefname{condition}{Condition}{Conditions}

\Crefname{figure}{Figure}{Figures}

\crefformat{equation}{\textup{#2(#1)#3}}
\crefrangeformat{equation}{\textup{#3(#1)#4--#5(#2)#6}}
\crefmultiformat{equation}{\textup{#2(#1)#3}}{ and \textup{#2(#1)#3}}
{, \textup{#2(#1)#3}}{ and \textup{#2(#1)#3}}
\crefrangemultiformat{equation}{\textup{#3(#1)#4--#5(#2)#6}}%
{ and \textup{#3(#1)#4--#5(#2)#6}}{, \textup{#3(#1)#4--#5(#2)#6}}{ and \textup{#3(#1)#4--#5(#2)#6}}

\Crefformat{equation}{#2Equation~\textup{(#1)}#3}
\Crefrangeformat{equation}{Equations~\textup{#3(#1)#4--#5(#2)#6}}
\Crefmultiformat{equation}{Equations~\textup{#2(#1)#3}}{ and \textup{#2(#1)#3}}
{, \textup{#2(#1)#3}}{ and \textup{#2(#1)#3}}
\Crefrangemultiformat{equation}{Equations~\textup{#3(#1)#4--#5(#2)#6}}%
{ and \textup{#3(#1)#4--#5(#2)#6}}{, \textup{#3(#1)#4--#5(#2)#6}}{ and \textup{#3(#1)#4--#5(#2)#6}}

\def\ds{\displaystyle}
\def\ns{\noalign{\smallskip}}

\crefdefaultlabelformat{#2\textup{#1}#3}

\newtheorem {theorem}{Theorem}[section]
\newtheorem {lemma}[theorem]{{\bf Lemma}}

\newtheorem {proposition}[theorem]{{\bf Proposition}}
\theoremstyle{remark}
\newtheorem {remark}{{\bf Remark}}[section]

\newtheorem {condition}{{\bf Condition}}[section]
\theoremstyle{definition}

\newtheorem {example}{{\bf Example}}[section]
\theoremstyle{plain} \numberwithin {equation}{section}

\newtheorem{problem}{Problem}
\numberwithin{assumption}{section}

\allowdisplaybreaks

\usepackage{graphicx}
\usepackage{ifpdf}
\ifpdf
\usepackage{epstopdf}
\fi

\usepackage{enumerate}

\def\deq{\mathop{\buildrel\Delta\over=}}

\usepackage{stmaryrd}

\usepackage{algorithm}
\usepackage{algorithmic}
\usepackage{graphicx}
\usepackage{subcaption} 

\begin{document}

\title{\bf An Inverse Source Problem for Semilinear Stochastic Hyperbolic Equations\footnote{
This work is supported by the NSF of China under grants 12401589 and 12025105, and by the Fundamental Research Funds for the Central Universities 2682024CX013.
}}

\author{
Qi L\"{u}
\footnote{School of Mathematics, Sichuan University, Chengdu, P. R. China. Email: lu@scu.edu.cn. }
~~~ and ~~~
Yu Wang
\footnote{School of Mathematics, Southwest Jiaotong University, Chengdu, P. R. China.
Email: yuwangmath@163.com.}
}

\date{}

\maketitle

\begin{abstract}

This paper investigates an inverse source problem for general semilinear stochastic hyperbolic equations. 
Motivated by the challenges arising from both randomness and nonlinearity, we develop a globally convergent iterative regularization method that combines Carleman estimate with fixed-point iteration. 
Our approach enables the reconstruction of the unknown source function from partial lateral Cauchy data, without requiring a good initial guess. 
We establish a new Carleman estimate for stochastic hyperbolic equations and prove the convergence of the proposed method in weighted spaces. 
Furthermore, we design an efficient numerical algorithm that avoids solving backward stochastic partial differential equations and is robust to randomness in both the model and the data. 
Numerical experiments are provided to demonstrate the effectiveness of the method.

\end{abstract}

\noindent{\bf Keywords}.  Semilinear stochastic hyperbolic equations, inverse source problem, numerical methods, regularization, Carleman estimate.

\section{Introduction}

Let $T>0$, and let $(\Omega,\mathcal{F},\mathbf{F},\mathbb{P})$ be a complete filtered probability space, where $\mathbf{F}=\{\mathcal{F}_t\}_{t\ge0}$ is the natural filtration generated by a one-dimensional standard Brownian motion $\{W(t)\}_{t\ge0}$, augmented by all $\mathbb{P}$-null sets in $\mathcal{F}$. Write $\mathbb{F}$ for the progressive $\sigma$-field with respect to $\mathbf{F}$. 
Let $\mathcal{H}$ be a Banach space. 
Denote by $L^2_{\mathbb{F}}(0,T;\mathcal{H})$ the Banach space consisting of all $\mathcal{H}$-valued $\mathbf{F}$-adapted processes $X(\cdot)$ satisfying $\mathbb{E}\bigl(|X(\cdot)|^2_{L^2(0,T;\mathcal{H})}\bigr)<\infty$, equipped with the canonical norm; by $L^\infty_{\mathbb{F}}(0,T;\mathcal{H})$ the Banach space of all $\mathcal{H}$-valued $\mathbf{F}$-adapted essentially bounded processes; 
and by $L^2_{\mathbb{F}}(\Omega;C([0,T];\mathcal{H}))$ the Banach space of all $\mathcal{H}$-valued $\mathbf{F}$-adapted continuous processes $X(\cdot)$ for which $\mathbb{E}\bigl(|X(\cdot)|^2_{C([0,T];\mathcal{H})}\bigr)<\infty$ (similarly, one can define $ L^2_{\mathbb{F}}(\Omega;H^{1}(0,T;\mathcal{H})) $).

Let $G$ be a bounded domain in $\mathbb{R}^n$ with a $C^2$ boundary $\Gamma = \partial G$. Let $(b^{j k})_{1\leq j, k \leq n} \in C^{1}(G; \mathbb{R}^{n})$ satisfy that $ b^{j k} = b^{k j}  $ for all $ j, k = 1, \cdots, n $ and for some constant $ s_{0} > 0 $, 
\begin{align*}
\sum_{j, k = 1}^{n} b^{j k} \xi_{j} \xi_{k} \geq s_{0} |\xi|^{2},   \text{ for all }   (x, \xi) \in G \times \mathbb{R}^{n}.
\end{align*}
Denote by $\nu(x) = (\nu^1(x), \nu^2(x), \cdots, \nu^n(x))$ the unit outward normal vector of $\Gamma$ at $x\in \Gamma$. Put
\begin{align}\label{eqGamma0}
\Gamma_0 \deq \Big\{ x\in \Gamma \;\Big|\;
\sum_{j,k=1}^nb^{jk}(x)\psi_{x_j}(x)\nu^{k}(x)
> 0 \Big\}.
\end{align}
where $\psi$ is given in Condition \ref{conB}. We  use the notation $ y_{x_i}= y_{x_i}(x)\equiv\frac{\partial y(x)}{\partial x_i} $,
where $x=(x_1,\ldots,x_n)\in\mathbb{R}^n$ and $x_i$ is the $i$-th coordinate.

Let $Q = (0,T) \times G$, $\Sigma = (0,T) \times \Gamma$ and $ \Sigma_{0} = (0,T) \times \Gamma_{0} $.
Consider the following semilinear stochastic hyperbolic equation:
\begin{align}\label{eqSemlinearStochasticHyperbolic}
\left\{
\begin{aligned}
& du_{t} - \sum_{j,k=1}^n(b^{jk}u_{x_j})_{x_k}dt=F(u, u_{t}, \nabla u) dt +  a u dW(t)&\mbox{ in }Q,\\
&   u_{t}(0) = 0 &\mbox{ in }G, \\
&   u(0) = u_{0} &\mbox{ in }G
,
\end{aligned}
\right.
\end{align}
where  $ a \in L^{\infty}_{{\mathbb{F}}}(0,T;L^{\infty}(G)) $ and $F: [0,T]\times\Omega\times G\times \mathbb{R}  \times\mathbb{R} \times\mathbb{R}^n\to\ \mathbb{R} $ satisfying the following assumptions:
\begin{condition}
\label{conF}
\begin{enumerate}[(1)]
\item For each $ (\eta, \rho, \zeta) \in \mathbb{R} \times \mathbb{R} \times \mathbb{R}^{n} $,  $ F(\cdot, \cdot, \cdot, \eta, \rho, \zeta) : [0,T]\times\Omega\times G  $ is $ L^{2}(G) $-valued $ \mathbf{F}$-adapted processes;
\item There exists a constant $ L $ such that for a.e. $(t,\omega,x)\in [0,T]\times\Omega\times G$ and any $(\eta_{i},\varrho_{i},\zeta_{i})\in \mathbb{R} \times\mathbb{R} \times\mathbb{R}^n$ ($i=1,2$),
\begin{align*}
|F(t,\omega,x,\eta_1,\varrho_1,\zeta_1)-F(t,\omega,x,\eta_2,\varrho_2,\zeta_2)|
\leq 
L(|\eta_1-\eta_2|+|\varrho_1-\varrho_2|+|\zeta_1-\zeta_2|_{\mathbb{R}^n})
;
\end{align*}
\item $ |F(\cdot,\cdot,\cdot,0,0,0)|\in L^2_{\mathbb{F}}(0,T;L^2(G)) $.
\end{enumerate}
\end{condition}

\smallskip

Stochastic hyperbolic equations play a fundamental role in modeling the vibrations of strings and membranes under random force, as well as wave propagation in random media (e.g., \cite{Cabana1970,Dalang2009,Funaki1983}). Their inverse problems have significant practical applications. A notable example is the international Solar and Heliospheric Observatory (SOHO) project \cite{Hellemans1996}, which seeks to reconstruct the Sun's internal structure by analyzing surface motion measurements \cite{Dalang2009}. This problem can be mathematically formulated as a series of inverse problems for stochastic hyperbolic equations, one of which is stated below:
\begin{problem} \label{prob1}
Determine the source function $u_{0}$ from the lateral Cauchy data $u \big|_{\Sigma}$ and $\dfrac{\partial u}{\partial \nu} \bigg|_{\Sigma_{0}}$.
\end{problem}

To study Problem \ref{prob1}, we introduce a technical assumption on $ (b^{j k})_{1\leq j, k \leq n}$:
\begin{condition}
\label{conB}
There exists a positive function $\psi(\cdot) \in C^2(\overline{G})$
satisfying the following:
\begin{enumerate}[(i)]
\item For some constant $\mu_0 > 0$ and any $ (x,\xi) \in     \overline{G} \times \mathbb{R}^n $, it holds that
\begin{align*}
	\sum_{j,k=1}^n\Big\{ \sum_{j',k'=1}^n\Big[
	2b^{jk'}(b^{j'k}\psi_{x_{j'}})_{x_{k'}} -
	b^{jk}_{x_{k'}}b^{j'k'}\psi_{x_{j'}} \Big] \Big\}\xi_{j}\xi_{k} \geq
	\mu_0 \sum_{j,k=1}^nb^{jk}\xi_{j}\xi_{k} .
\end{align*}

\item  There is no critical point of $\psi(\cdot)$ in
$\overline{G}$, i.e.,
\begin{align*}
	\min_{x\in \overline{G} }|\nabla \psi(x)| > 0.
\end{align*}
\end{enumerate}
\end{condition}
\begin{remark}
Condition \ref{conB} ensures the existence of a suitable weight function for establishing the necessary Carleman estimate for  linear stochastic hyperbolic equations. This condition represents a pseudoconvexity requirement (cf. \cite[XXVIII]{Hoermander2009}). In particular, when the matrix $(b^{jk}(\cdot))_{1\leq j,k\leq n}$ coincides with the $n\times n$ identity matrix $I_n$, the function $\varphi(x) = |x - x_0|^2$ with $x_0 \in \mathbb{R}^n \setminus \overline{G}$ satisfies Condition \ref{conB} with the constant $s_0 = 4$. 
\end{remark}

Observe that if  $\psi(\cdot)$ satisfies Condition \ref{conB}, then for any constants $a \geq 1$ and $b \in \mathbb{R}$, the rescaled function $\tilde{\psi} = a\psi + b$ also satisfies this condition with $\mu_0$ replaced by $a\mu_0$. Consequently, we may select $\psi$, $\mu_0$, and positive constants  $c_0>0$, $c_1>0$ such that the following strengthened condition holds:
\begin{condition}
\label{conPsi2}
\begin{align*}
\left\{
\begin{aligned}
& \frac{1}{4}\sum_{j,k=1}^n b^{jk}(x)\psi_{x_j}(x)\psi_{x_k}(x) \geq
R^2_1 \deq\max_{x\in\overline G}\psi(x),\quad \forall x\in \overline G,\\ 
&   c_1<\frac{ R_1}{\sqrt{2}T},\quad \mu_0 - 4c_1 -c_0 > 0.
\end{aligned}
\right.
\end{align*}
\end{condition}
%


Inverse problems for deterministic hyperbolic equations have been extensively studied (see, e.g., \cite{Belishev2007,Bellassoued2017,Klibanov2013,Klibanov2021,Lassas2015,Symes2009} and references therein).  
In contrast, inverse problems for stochastic hyperbolic  equations remain far from being well studied. The reason is that, compared with inverse problems for deterministic hyperbolic equations, the presence of randomness and uncertainty in both the model and the data introduces significant additional challenges in the study of inverse problems for stochastic hyperbolic equations. Those intrinsic difficulties caused by randomness and uncertainty make it difficult to directly extend methods developed for deterministic hyperbolic equations to the stochastic setting. 
As a result, the explicit determination and reconstruction of these stochastic sources and potentials remain largely unexplored, representing a critical gap in current research. To our best knowledge, only six works concerning the explicit determination and reconstruction of these sources and potentials \cite{Feng2022,Zhang2008,Lue2013a,Lue2015a,Yuan2015,Dou2024a}. In \cite{Feng2022}, significant progress  on determining the {\it expectation} and {\it variance} of {\it random sources} for stochastic wave equations has been made.  In \cite{Zhang2008,Lue2013a,Lue2015a,Yuan2015}, it is proved that the unknown source function was determined from partial boundary measurements for semilinear stochastic hyperbolic equations. Nevertheless, those works considered the case of homogeneous Dirichlet boundary conditions and did not provide a method for reconstructing the source function. 
Recently, in \cite{Dou2024a} the inverse Cauchy problem for linear stochastic hyperbolic equations was addressed, and numerical examples were provided using a kernel-based learning method, with {\it full} boundary measurements.

Classical approaches to nonlinear inverse problems primarily rely on optimization techniques, which are local in nature and work well when a good initial guess is available. 
However, solving nonlinear problems without a reliable initial guess remains a challenging and important issue. 
One widely recognized global strategy is convexification, which employs a suitable Carleman weight function to convexify the functional, with the justification provided by Carleman estimates.
Since its introduction in \cite{Klibanov1995}, the convexification method has seen various adaptations \cite{Klibanov2016,Klibanov2015}. 
Nevertheless, a major weak point of convexification is its high computational cost \cite{Nguyen2022}. 
To address this, and inspired by \cite{Dang2024,Nguyen2022}, we propose a new method that combines fixed-point iteration with suitable Carleman estimates to achieve global convergence for nonlinear inverse problems. 
Here, ``global'' means that our approach does not require an initial guess close to the true solution.

Our approach to solving \cref{prob1} is to construct a map $ \mathfrak{G}(u_{n}) \! = \!u_{n+1}\! =\! \arg\min J_{n}(u; u_{n}) $ via a Tikhonov functional (see \cref{eqFunctional}), where the fixed point of $\mathfrak{G}$ corresponds to the solution of the desired inverse problem. 
Starting from an arbitrary initial guess $u_{0}$, we recursively define $u_{n}$ for $n\in\mathbb{N}$. 
The main result of this paper, \cref{thmConverge}, shows that this iterative sequence converges to the true solution in a suitable weighted space with an appropriate Carleman weight, and that the stability with respect to noise in the data is of H\"{o}lder type.

Finally, we present a numerical \cref{algInverseProblem} for solving the inverse \cref{prob1}.
Compared with deterministic inverse problems, the stochastic counterparts introduce  additional challenges: the lateral Cauchy data (see \cref{rkForwardBoundary}) and the gradients involved in optimization (see \cref{rkDiscreteFunctional}) are themselves stochastic processes, making classical deterministic optimization methods such as the conjugate gradient method unstable or inefficient (see \cref{fig:compareSGD_Adam}).
To address this, we employ the adaptive moment estimation (Adam) optimizer \cite{Kingma2014} for stochastic optimization and utilize automatic differentiation \cite{Paszke2017} to efficiently compute gradients, thereby avoiding the need to solve backward stochastic partial differential equations (see, e.g., \cite{Dou2022,Wang2025a}).
This approach ensures robust and efficient convergence even in the presence of randomness and high-dimensional parameter spaces; see \Cref{secNumerical} for further details.

This work is not a straightforward extension of previous studies on inverse problems for semilinear deterministic hyperbolic equations \cite{Nguyen2022,Dang2024} or linear stochastic hyperbolic equations \cite{Dou2024a}. Below, we discuss the distinctions and connections with these works in detail.

\smallskip

{\bf Comparison with \cite{Dang2024,Nguyen2022}}. 
In this work, we study a semilinear stochastic hyperbolic equation with general coefficients, using only partial boundary observation data. Unlike deterministic inverse problems, the stochastic nature of our setting introduces significant challenges, particularly due to the fact that Brownian sample paths are almost surely non-differentiable in time. To overcome these difficulties, we develop:
\begin{itemize}
\item A new Carleman estimate tailored to stochastic hyperbolic equations (see \cref{thmCarlemanEstimate});
\item A modified Tikhonov functional (see \cref{eqFunctional}) designed to handle the irregularity of stochastic processes.
\end{itemize}
Furthermore, stochastic setting poses additional structural challenges, as the  solution spaces typically lack compact embedding properties. Consequently, the techniques employed in \cite{Dang2024,Nguyen2022} for deterministic problems are not directly applicable here.

\smallskip

{\bf Comparison with \cite{Dou2024a}}.  
While \cite{Dou2024a} investigates the inverse Cauchy problem for linear stochastic hyperbolic equations, our work focuses on the semilinear case. The introduction of nonlinearity necessitates fundamentally different approaches at both theoretical and computational levels:
\begin{itemize}
\item Unlike the direct regularization method employed in \cite{Dou2024a}, we develop an iterative regularization scheme based on fixed-point iteration for \cref{eqSemlinearStochasticHyperbolic}.
\item Our Carleman estimate (see \cref{rkCarlemanEstimate}) differs  from that in \cite{Dou2024a}, as it must account for nonlinear effects in proving iterative convergence.
\item While \cite{Dou2024a} utilizes mesh-free kernel-based learning, we propose \cref{algInverseProblem} based on discretization of the Tikhonov functional \cref{eqDiscreteFunctional}.
\item Whereas \cite{Dou2024a} needs measurements of both  $u$ and  $\dfrac{\partial u}{\partial \nu}$ on the entire lateral boundary, our approach solves \cref{prob1} without requiring the boundary data for $ \dfrac{\partial u}{\partial \nu} $ on $(0,T) \times (\Gamma \setminus \Gamma_{0})$.
\end{itemize}

In a word, our new method offers several advantages:
\begin{enumerate}
\item It is specifically designed for stochastic PDEs, effectively handling randomness and circumventing the challenges of solving backward stochastic PDEs.
\item It does not require a good initial guess.
\item It exhibits a fast convergence rate.
\item It enables the reconstruction of the source function from partial boundary measurement.
\item It  does not require the special structure on the nonlinearity $F$.
\end{enumerate}

We end this section by highlighting an important point.

\begin{remark} 
Although we assume in \cref{conF} (2) that $ F(t, \omega, x, \cdot, \cdot, \cdot) $
is Lipschitz continuous, this may not hold in practice. 
In such cases, one needs to know a constant $M$ in advance such that the true solution $u^{*}$ of equation \cref{eqSemlinearStochasticHyperbolic} satisfies 
$ |u^{*}|_{L^{\infty}_{\mathbb{F}}(\Omega, W^{1,\infty}(0,T;W^{1,\infty}(G)))} \leq M $.
This does not weaken our results, since $M$ can be arbitrarily large. 
Then we construct a $C^{1}$ function $ \chi : \mathbb{R} \times \mathbb{R} \times \mathbb{R}^n \rightarrow \mathbb{R} $
satisfying $\chi=1$ in $ \{(\eta,\varrho,\zeta)\in \mathbb{R} \times\mathbb{R} \times\mathbb{R}^n \mid |\eta|+|\varrho|+|\zeta|_{\mathbb{R}^n}\leq M\} $
and $\chi=0$ in  $ \{(\eta,\varrho,\zeta)\in \mathbb{R} \times\mathbb{R} \times\mathbb{R}^n \mid |\eta|+|\varrho|+|\zeta|_{\mathbb{R}^n}>2M\} $.
We then solve \cref{prob1} for equation \cref{eqSemlinearStochasticHyperbolic} with $F$ replaced by $\chi F$. 
Numerical  \cref{ex2} demonstrates that this approach is effective. 
\end{remark}

The rest of this paper is organized as follows. In \Cref{secCarleman}, we establish a Carleman estimate for a linear stochastic hyperbolic equation.  
In  \Cref{secRegularization}, we propose an iterative regularization scheme to solve \cref{prob1} and establish its convergence properties. 
Finally, in \Cref{secNumerical} we present the numerical algorithm  for \cref{prob1} along with numerical experiments.

\section{Carleman estimate for a linear stochastic hyperbolic equation}
\label{secCarleman}

In this section, we establish a global Carleman estimate for the following linear stochastic hyperbolic equation:
\begin{equation}\label{eqLinearStochasticHyperbolic}
\begin{cases}
\ds d z_{t} - \sum_{j,k=1}^n(b^{jk}z_{x_j})_{x_k}dt=\Upsilon  dt +  a z dW(t)&\mbox{ in }Q,\\
\ns\ds   z   = 0 &\mbox{ on } \Sigma, \\
\ns\ds    \frac{\partial z}{\partial \nu} = 0  & \mbox{ on } \Sigma_{0}.
\end{cases} 
\end{equation}
\begin{theorem}
\label{thmCarlemanEstimate}
Assume conditions \cref{conB,conPsi2,conB} hold. Then there exist positive constants $ C $ and $ \lambda_{0} $ such that for all $ \lambda \geq \lambda_{0} $, the solution $z$ to \cref{eqLinearStochasticHyperbolic} with initial condition $ z_{t}(0) = 0 $ satisfies the following Carleman estimate:
\begin{align}\label{eqCar1}
\mathbb{E} \int_{Q} \theta^{2} (\lambda z_{t}^{2} + \lambda |\nabla z|^{2} + \lambda^{3} z^{2} ) d x d t 
\leq 
C \mathbb{E} \int_{G} \theta^{2} (\lambda |\nabla z|^{2} + \lambda^{3} |z|^{2}) \big|_{t = T} d x 
+ C \mathbb{E} \int_{Q} \theta^{2} |\Upsilon |^{2} d x d t,
\end{align}
where the weight function $\theta$ is defined by
\begin{align}
\label{eqTheta}
\theta = e^{\ell}, \quad \quad \ell(t,x) = \lambda (\psi(x) - c_{0} t^{2}).
\end{align}
\end{theorem}
\begin{remark}
\label{rkCarlemanEstimate}
We emphasize that \cref{thmCarlemanEstimate} employs a different Carleman estimate than \cite{Dou2024a}, where the weight function $ \theta $ is chosen with  $\theta_{t}(0) = 0$  to ensure iteration convergence (\cref{thmConverge}).
Furthermore, since our assumption on the measurement noise is given by \cref{eqNoise}, it is sufficient here to derive the Carleman estimate under homogeneous boundary conditions.
\end{remark}

Here and in what follows, we denote by $C$ a generic positive constant that depends only on the domain $G$, time horizon $T$, and coefficients $(b^{jk}(\cdot))_{1\leq j,k\leq n}$ and $a$, whose exact value may change from line to line.

To prove \cref{thmCarlemanEstimate}, we recall the following known result.

\begin{lemma}\cite[Lemma 10.18]{Lue2021a}
\label{lemmaFundamentalIdentity}
Let $\phi \in C^1((0,T)\times\mathbb{R}^n)$, $b^{jk}=b^{kj}\in
C^2(\mathbb{R}^n)$ for $j,k=1,2,\cdots,n$, 
and $\ell,\Psi \in C^2((0,T)\times\mathbb{R}^n)$. 
Let $\varphi$ be an $H^2(G)$-valued, $\mathbf{F}$-adapted process such that $\varphi_t$ is an $L^2(G)$-valued It\^o process. 
Set $\theta = e^\ell$ and $w=\theta \varphi$. 
Then, for a.e.     $x\in G$ and ${\mathbb{P}}$-a.s. $\omega \in \Omega$,
\begin{align}
\notag
\label{hyperbolic2}  
&   \theta \Big ( -2\phi\ell_t w_t + 2\sum_{j,k=1}^n b^{jk}\ell_{x_j} w_{x_k} + \Psi w \Big )
\Big[ \phi d \varphi_t - \sum_{j,k=1}^n (b^{jk}\varphi_{x_j})_{x_k} dt \Big]  
\\ \notag
& \quad+\sum_{j,k=1}^n \Big[ \sum_{j',k'=1}^n \Big (
2b^{jk}b^{j'k'}\ell_{x_{j'}}w_{x_j}w_{x_{k'}} -
b^{jk}b^{j'k'}\ell_{x_j} w_{x_{j'}}w_{x_{k'}}
\Big )   
- 2 \phi b^{jk}\ell_t w_{x_j} w_t 
+ \phi b^{jk}\ell_{x_j} w_t^2  
\\ \notag
&\quad  \quad \quad \quad 
+ \Psi b^{jk}w_{x_j} w - \Big({\cal A}\ell_{x_j} +
\frac{\Psi_{x_j}}{2}\Big)b^{jk}w^2 \Big]_{x_k} dt
\\ \notag
&\quad +d\Big\{ \phi\sum_{j,k=1}^n b^{jk}\ell_t w_{x_j} w_{x_k}-
2\phi\sum_{j,k=1}^n b^{jk}\ell_{x_j}w_{x_k}w_t  + \phi^2\ell_t w_t^2
- \phi\Psi w_t w 
+ \Big[ \phi {\cal A}\ell_t +
\frac{1}{2}(\phi\Psi)_t\Big]w^2 \Big\}   
\\ \notag
&   =   \Big[(\phi^2\ell_{t})_t + \sum_{j,k =1}^n (\phi
b^{jk}\ell_{x_j})_{x_k} - \phi\Psi \Big]w_t^2   d t
- 2\sum_{j,k=1}^n \big[(\phi
b^{jk}\ell_{x_k})_t +
b^{jk}(\phi\ell_{t})_{x_k}\big]w_{x_j}w_t   d t   
\\
& \quad 
+\sum_{j,k=1}^n c^{jk}w_{x_j}w_{x_k} d t + \mathcal{B}w^2 d t
+ \Big(   -2\phi\ell_t w_t 
+ 2\sum_{j,k=1}^n b^{jk}\ell_{x_j}w_{x_k} + \Psi w
\Big)^2 dt + \phi^2\theta^2 \ell_t(d\varphi_t)^2,  
\end{align}
where \vspace{-2mm}
\begin{align}\label{AB1}  
\left\{\begin{aligned}
& {\cal A}\deq\phi (\ell_t^2-\ell_{tt})-\sum_{j,k=1}^n
(b^{jk}\ell_{x_j}\ell_{x_k}-b^{jk}_{x_k}\ell_{x_j}
-b^{jk}\ell_{x_jx_k})-\Psi,\\  
& {\cal B}\deq{\cal A}\Psi+(\phi {\cal A}\ell_t)_t -  \sum_{j,k=1}^n ({\cal A}
b^{jk}\ell_{x_j})_{x_k}  + \dfrac 12
\Big[\big(\phi\Psi\big)_{tt} - \sum_{j,k=1}^n 
\big(b^{jk}\Psi_{x_j}\big)_{x_k}\Big],\\  
& c^{jk}\deq \big(\phi b^{jk}\ell_t\big)_t  + 
\sum_{j',k'=1}^n  \Big[2b^{jk'}(b^{j'k}\ell_{x_{j'}})_{x_{k'}}
- \big(b^{jk}b^{j'k'}\ell_{x_{j'}}\big)_{x_{k'}}\Big]  +  \Psi
b^{jk}.
\end{aligned}
\right.
\end{align}
\end{lemma}

Throughout this paper, for $\lambda\in\mathbb{R}$ and $p\in\mathbb{N}$, we denote by $O(\lambda^p)$ any function of order $\lambda^p$ for large $\lambda$.  

\begin{proof}[Proof of \cref{thmCarlemanEstimate}]

Choose\vspace{-2mm}
\begin{equation}
\label{eqPsi}
\varphi=z, \quad  \phi \equiv 1, \quad  \Psi = \ell_{tt} + \sum_{j,k=1}^n (b^{j k} \ell_{x_j})_{x_k} - c_{0} \lambda
,
\end{equation}
in \cref{lemmaFundamentalIdentity}.
We now estimate the terms in \cref{hyperbolic2} one by one.

We first note that 
\begin{align}
\label{eqEllEstimate}
\ell_{t} = - 2 c_{1} \lambda t, \quad 
\ell_{tt} = - 2 c_{1} \lambda, \quad
\ell_{x_{j}} = \lambda \psi_{x_{j}}, \quad
\ell_{t x_{j}} = 0,\quad j=1,\cdots,n
.
\end{align}
From \cref{eqEllEstimate,eqPsi}, we derive that
\begin{align}
\label{eqwt2}
\Big[(\phi^2\ell_{t})_t + \sum_{j,k =1}^n (\phi
b^{jk}\ell_{x_j})_{x_k} - \phi\Psi \Big]w_t^2 
= 
c_{0} \lambda w_t^2,
\end{align}
and \vspace{-2mm}
\begin{align}
\label{eqwjwt}
- 2\sum_{j,k=1}^n \big[(\phi b^{jk}\ell_{x_k})_t + b^{jk}(\phi\ell_{t})_{x_k}\big]w_{x_j}w_t  = 0.
\end{align}
Furthermore, combining \cref{eqEllEstimate,eqPsi,conB} yields
\begin{align}\label{eqvivj}\notag
\sum_{j,k=1}^n c^{jk}w_{x_j}w_{x_k} 
&  = 
\sum_{j,k=1}^n \Big\{ (b^{jk}\ell_t)_t \!+\! \sum_{j',k'=1}^n
\big[ 2b^{jk'}(b^{j'k}\ell_{x_{j'}})_{x_{k'}} \!-\!
(b^{jk}b^{j'k'}\ell_{x_{j'}})_{x_{k'}} \big] \!+\!
\Psi b^{jk} \Big\}w_{x_j} w_{x_k} 
\\ \notag
& = 
\sum_{j,k=1}^n \Big\{  2 b^{jk}\ell_{tt} + \sum_{j',k'=1}^n
\big[ 2b^{jk'}(b^{j'k}\ell_{x_{j'}})_{x_{k'}}-
b^{jk}_{x_{k'}}b^{j'k'}\ell_{x_{j'}} \big]
- c_{0} \lambda  b^{jk} \Big\}w_{x_j} w_{x_k} 
\\
& \geq 
\lambda (\mu_0 -4c_1 - c_0)\sum_{j,k=1}^nb^{jk}w_{x_j}w_{x_k}
.
\end{align}

From \cref{eqEllEstimate}, we obtain\vspace{-4mm}
\begin{align}
\label{eqA}
\mathcal{A} 
& = 
\lambda^{2} \Big( 4 c_{1}^{2}   t ^{2} - \sum_{j, k=1}^{n} b^{j k} \psi_{x_{j}} \psi_{x_{k}} \Big) + \mathcal{O} (\lambda)
\\ \notag
(\mathcal{A} \ell_{t})_{t} 
& =
- 24 c_{1}^{3} \lambda^{3} t^{2} + 2 c_{1} \lambda^{3} \sum_{j, k=1}^{n} b^{j k} \psi_{x_{j}} \psi_{x_{k}}   + \mathcal{O} (\lambda^{2})
,
\end{align}
and \vspace{-2mm}
\begin{align*} 
- \sum_{j, k=1}^{n} ( \mathcal{A} b^{j k} \ell_{x_{j}})_{x_{k}} 
&
= 
\lambda^{3} \sum_{j, k=1}^{n} \Big\{
- c_{1}^{2}  (b^{j k} \psi_{x_{j}})_{x_{k}} 
+ \sum_{j', k'=1}^{n}   (b^{j k} \psi_{x_{j}} \psi_{x_{k}})_{x_{k'}} b ^{j' k'} \psi_{x_{j'}}
\\
& \quad \quad \quad \quad \quad 
+\sum_{j', k'=1}^{n}   b^{j k} \psi_{x_{j}} \psi_{x_{k}} (b ^{j' k'} \psi_{x_{j'}})_{x_{k'}}
\Big\}
+ \mathcal{O} (\lambda^{2})
.
\end{align*}
Combining this with \cref{AB1,eqPsi,eqEllEstimate} yields
\begin{align}
\notag
\label{eqBestimate}
\mathcal{B}  &  =  (4c_1+c_0)\lambda^3 \sum_{j,k=1}^nb^{jk}\psi_{x_j}\psi_{x_k} +
\lambda^3\sum_{j,k=1}^n\sum_{j',k'=1}^nb^{jk}\psi_{x_j}
\big(b^{j'k'}\psi_{x_{j'}}\psi_{x_{k'}}\big)_{x_k} 
\\
& \quad   - 4 (8c_1^3 + c_0c_1^2)\lambda^3 t^2 + O(\lambda^2).
\end{align}
Under \cref{conB,conPsi2}, we establish from \cref{eqBestimate} the lower bound
\begin{align*}
\notag
\mathcal{B} 
& \geq (4c_1+c_0 + \mu_{0})\lambda^3 \sum_{j,k=1}^nb^{jk}\psi_{x_j}\psi_{x_k}  
- 4 (8c_1^3 + c_0c_1^2)\lambda^3 t^2 + O(\lambda^2)
\\ \notag
& \geq 4 (4c_1+c_0 + \mu_{0})\lambda^3  R_{1}^{2}
- 4 (8c_1^3 + c_0c_1^2)\lambda^3 T^2 + O(\lambda^2)
\\
& 
\geq  \big(4c_1+c_0\big)\lambda^3 \big(4R_1^2 - 8 c_1^2T^2\big) +
O(\lambda^2)
.
\end{align*}
By \cref{conPsi2}, there exists a positive constant $ \lambda_{1} $ such that for all $ \lambda \geq \lambda_{1} $, we have
\begin{align}
\label{eqB}
\mathcal{B} w^{2} \geq C \lambda^{3} w^{2}
.
\end{align}

Since $ \ell_{t}(0) = 0 $ and $ w_{t}(0) = \theta(0) z_{t} (0) + \theta_{t} (0) z(0) = 0 $, applying \cref{eqEllEstimate,eqPsi,eqA,conPsi2,conB} and the Cauchy-Schwarz inequality yields for any $ \varepsilon > 0 $:
\begin{align*}
& \mathbb{E} \int_{0}^{T} \int_{G}  d\Big\{ \phi\sum_{j,k=1}^n b^{jk}\ell_t w_{x_j} w_{x_k}\!-\!
2\phi\sum_{j,k=1}^n b^{jk}\ell_{x_j}w_{x_k}w_t  \!+\! \phi^2\ell_t w_t^2
\!-\! \phi\Psi w_t w 
\!+\! \Big[ \phi {\cal A}\ell_t \!+\!
\frac{1}{2}(\phi\Psi)_t\Big]w^2 \Big\}  
\\
& =
\mathbb{E} \int_{G} \Big(\sum_{j,k=1}^n b^{jk}\ell_t w_{x_j} w_{x_k}- 2\sum_{j,k=1}^n b^{jk}\ell_{x_j}w_{x_k}w_t  + \ell_t w_t^2 - \Psi w_t w +  \mathcal{A} \ell_{t} w^{2} \Big)\bigg|_{t = T} d x
\\
& = 
\mathbb{E} \int_{G} \Big[
- 2 c_{1} T \lambda \Big(\sum_{j,k=1}^n b^{jk}   w_{x_{j}} w_{x_{k}} + w_{t}^{2} \Big)
- 2\sum_{j,k=1}^n b^{jk}\ell_{x_j}w_{x_k}w_t    - \Psi w_t w 
\\
& \quad \quad \quad \quad 
- 2 c_{1} T \lambda^{3} \Big( 4 c_{1}^{2}   t ^{2} - \sum_{j, k=1}^{n} b^{j k} \psi_{x_{j}} \psi_{x_{k}} \Big)   w^{2} 
+ \mathcal{O}(\lambda^{2}) w^{2} \Big]\bigg|_{t = T} d x
\\
& \leq 
\mathbb{E} \int_{G} \big[
- 2 c_{1} T \lambda  ( s_{0}  |\nabla w|^{2} + w_{t}^{2}  )
+ \varepsilon \lambda |w_{t} |^{2} 
+ C(\varepsilon) \lambda (|\nabla w|^{2} + w^{2})
\\
& \quad \quad \quad \quad 
- 2 c_{1} T \lambda^{3}  ( 4 c_{1}^{2}   T ^{2} -  4 R_{1}^{2} )  w^{2} 
+ \mathcal{O}(\lambda^{2}) w^{2} \big ]\Big|_{t = T} d x
.
\end{align*}
Choosing $ \varepsilon < 2c_{1} T $ and noting $ c_{1} T < R_{1} $, there exists $ \lambda_{2} >0 $ such that for all $ \lambda \geq \lambda_{2} $, we have
\begin{align}
\label{eqIntD} \notag
& \mathbb{E} \int_{0}^{T} \int_{G}  d\Big\{ \phi\sum_{j,k=1}^n b^{jk}\ell_t w_{x_j} w_{x_k}\!-\!
2\phi\sum_{j,k=1}^n b^{jk}\ell_{x_j}w_{x_k}w_t  \!+\! \phi^2\ell_t w_t^2
\!-\!\phi\Psi w_t w 
\!+\! \Big[ \phi {\cal A}\ell_t \!+\!
\frac{1}{2}(\phi\Psi)_t\Big]w^2 \Big\}  
\\
& \leq 
C   \mathbb{E} \int_{G} (\lambda^{3} |w(T)|^{2} + \lambda | \nabla w(T)|^{2}) d x
.
\end{align}

Integrating \cref{hyperbolic2} over $Q$, taking mathematical expectation of the integral, and noting \cref{eqB,eqvivj,eqwt2,eqwjwt,eqIntD}, for all $ \lambda \geq \max\{\lambda_{1}, \lambda_{2} \}  $, we get that
\begin{align}\notag
\label{bhyperbolic31}
&  \mathbb{E} \int_Q \theta\Big\{ \Big( -2\ell_t w_t + 2
\sum_{j,k=1}^n b^{jk}\ell_{x_j}w_{x_k} +
\Psi w \Big) \Big[du_t- \sum_{j,k=1}^n \big(b^{jk}\varphi_{x_j}\big)_{x_k}dt \Big]  \Big\}dx 
\\ \notag
&\quad + \lambda
\mathbb{E} \int_{\Sigma}\sum_{j,k=1}^n \sum_{j',k'=1}^n \big(
2b^{jk}b^{j'k'}\psi_{x_{j'}}w_{x_j} w_{x_{k'}}   - 
b^{jk}b^{j'k'}\psi_{x_j} w_{x_{j'}}w_{x_{k'}} \big)\nu_k d \Gamma d t  
\\\notag
& \quad 
+ C \mathbb{E} \int_{G} \theta^{2} (\lambda |\nabla z|^{2} + \lambda^{3} |z|^{2}) \big|_{t = T} d x 
-  \mathbb{E} \int_Q  \theta \ell_t |d z_t|^2 d x
\\ \notag
& \geq C \mathbb{E}\int_Q \theta^2\Big[\big( \lambda |z_t|^2 + \lambda
|\nabla z|^2 \big) 
+   \lambda^3 |z|^2 \Big]dxdt 
\\
& \quad 
+ \mathbb{E}\int_Q\Big(-2\ell_t w_t +
2\sum_{j,k=1}^nb^{jk}\ell_{x_j}w_{x_k} + \Psi w\Big)^2
dxdt
,
\end{align}
where we used the fact that $w = \theta z$ and $ w = z = 0 $ on $ \Sigma $.
Noting that $ \frac{\partial z}{\partial \nu} = 0 $ on $ \Sigma_{0} $ and \cref{eqGamma0}, we find that 
\begin{align} \notag
\label{bhyperbolic32}
& \mathbb{E}\int_{\Sigma}\sum_{j,k=1}^n\sum_{j',k'=1}^n\Big(
2b^{jk}b^{j'k'}\psi_{x_{j'}}w_{x_j} w_{x_{k'}} -
b^{jk}b^{j'k'}\psi_{x_j} w_{x_{j'}}w_{x_{k'}}
\Big)\nu^k d \Gamma d t   
\\ \notag
&  =  \mathbb{E}\int_{\Sigma}\Big( \sum_{j,k=1}^nb^{jk}\nu^j \nu^k
\Big)\Big( \sum_{j',k'=1}^nb^{j'k'}\psi_{x_{j'}}\nu^{k'}
\Big)\theta^2\Big|\frac{\partial z}{\partial
\nu}\Big|^2d \Gamma d t  
\\ \notag
&  \leq \mathbb{E}\int_{\Sigma_0}\Big( \sum_{j,k=1}^nb^{jk}\nu^j
\nu^k \Big)\Big( \sum_{j',k'=1}^nb^{j'k'}\psi_{x_{j'}}\nu^{k'}
\Big)\theta^2\Big|\frac{\partial z}{\partial \nu}\Big|^2d \Gamma d t 
\\
& =
0
.
\end{align}
From \cref{eqLinearStochasticHyperbolic}, we infer that 
\begin{align}
\label{bhyperbolic3}
- \mathbb{E} \int_Q  \theta \ell_t |d z_t|^2 d x
\leq 2 c_{1} \lambda  T |a|_{L^{\infty}_{{\mathbb{F}}}(0,T;L^{\infty}(G))} \mathbb{E} \int_{Q} \theta^{2} |z|^{2} d x d t,
\end{align}
and \vspace{-4mm}
\begin{align}
\notag
\label{bhyperbolic4}
& \mathbb{E} \int_Q \theta\Big\{ \Big( -2\ell_t w_t + 2
\sum_{j,k=1}^n b^{jk}\ell_{x_j}w_{x_k} +
\Psi w \Big) \Big[dz_t- \sum_{j,k=1}^n\!\big(b^{jk}z_{x_j}\big)_{x_k}dt \Big]  \Big\}dx  
\\ 
& 
\leq     \mathbb{E}\int_Q \theta^2 |\Upsilon |^{2} dxdt 
+ \mathbb{E} \int_Q \Big ( - 2\ell_t w_t + \sum_{j,k=1}^n
b^{jk}\ell_{x_j} w_{x_k} +  \Psi w \Big )^2dxdt.
\end{align}

Combining \cref{bhyperbolic31,bhyperbolic32,bhyperbolic4,bhyperbolic3}, choosing $ \lambda_{0} \geq \max\{C, \lambda_{1}, \lambda_{2}\} $, we obtain the desired Carleman estimate \cref{eqCar1}.
\end{proof}

\section{A convergence result in determining the source}
\label{secRegularization}

Denote $ \mathcal{H}^{p} = L^{2}_{\mathbb{F}}(0,T; H^{p}(G)) \cap L^{2}_{\mathbb{F}} (\Omega; H^{1}(0,T; H^{1}(G))) $ for $ p \in \mathbb{N} $. 
Define an operator $ \mathcal{P} : \mathcal{H}^{2} \to L^{2}_{\mathbb{F}}(0,T; L^{2}(G)) $ by as follows:\vspace{-2mm}
\begin{align*}
(\mathcal{P} \varphi)(t,x) 
\deq & \varphi_{t}(t, x)
- \int_0^t  \sum_{j, k=1}^n\big(b^{j k}(s, x) \varphi_{x_{j}}(s, x)\big)_{x_{k}}  d s
\\
& 
-\int_0^t a(s, x) \varphi(s, x) d W(s), \quad \text{${\mathbb{P}}$-a.s.,}\quad \forall\,\varphi \in \mathcal{H},\;(t,x) \in Q.
\end{align*}
The admissible set is given by 
\begin{align}
\label{eqAdmissibleSet}
\mathcal{U} \deq \Big\{ \varphi \in \mathcal{H}^{2} \;\Big|\; \mathcal{P} \varphi \in L^{2}_{\mathbb{F}}(\Omega; H^{1}(0,T; L^{2} (G))), 
\quad \varphi_{t}(0) = 0,
\quad  \varphi |_{\Sigma} = f,  
\quad  \frac{\partial \varphi}{\partial \nu} \Big|_{\Sigma_{0}} = g
\Big\}.
\end{align}
For any solution $u$ to \eqref{eqSemlinearStochasticHyperbolic}, we have:
$$
\mathcal{P} u = \int_{0}^{t} \mathcal{F}(u) d s,\quad \mbox{ where }\; \mathcal{F}(u) \deq F(u, u_{t}, \nabla u).
$$
Under assumption \eqref{conF}, we obtain $ \mathcal{F}(u) \in  L^{2}_{\mathbb{F}}(\Omega; H^{1}(0,T; L^{2} (G)))$, which implies $ u \in \mathcal{U} $. 
Consequently, the admissible set $ \mathcal{U} \neq \emptyset $.

Let $ \theta$ be the weight function defined in \eqref{eqTheta}. We introduce the following weighted Sobolev spaces:\vspace{-2mm}
\begin{align*}
& L^{2, w}_{\mathbb{F}}(\Omega; H^{1}(0,T; L^{2}(G)))
\\
& \deq 
\Big\{ \varphi \in L^{2}_{\mathbb{F}}(\Omega; H^{1}(0,T; L^{2}(G))) \;\Big|\; 
\mathbb{E} \int_{0}^{T} \int_{G} \theta^{2} | \varphi|^{2} d x d t 
+ \mathbb{E} \int_{0}^{T} \int_{G} \theta^{2} | \varphi_{t} |^{2} d x d t 
< \infty \Big\} 
\end{align*}
and\vspace{-4mm}
\begin{align*}
L^{2, w}_{\mathbb{F}}(0,T; H^{p}(G)) 
\deq 
\Big\{ \varphi \in L^{2}_{\mathbb{F}}(0,T; H^{p}(G)) \;\Big|\;
\sum_{|\alpha|=0}^{p} \mathbb{E} \int_{0}^{T} \int_{G} \theta^{2} |D_{\alpha} \varphi|^{2} d x d t < \infty
\Big\},
\end{align*}
where $ \alpha = (\alpha_{1}, \cdots, \alpha_{n}) $ is a multi-index with $ |\alpha| = \sum\limits_{i=1}^{n} \alpha_{i} $ and $ D_{\alpha} = \frac{\partial^{|\alpha|}}{\partial \alpha_{1} \cdots \partial \alpha_{n}} $ denotes the weak derivative operator. 
Let\vspace{-2mm}
\begin{align*}
\mathcal{H}^{p}_{\lambda, c_{0}} \deq L^{2, w}_{\mathbb{F}}(0,T; H^{p}(G)) \cap L^{2, w}_{\mathbb{F}} (\Omega; H^{1}(0,T; H^{1}(G)))
,
\end{align*}
where $ \lambda, c_{0} $ are the parameters from \eqref{eqTheta}.
Each of these spaces, when equipped with its canonical norm, forms a Banach space.

Let the integral operator $\mathbf{I}$ be defined by\vspace{-2mm}
$$ \mathbf{I} f(t) = \int_{0}^{t} f(s) d s \quad\mbox{ for }\; t\in [0,T].$$

Let us construct a sequence $ \{u_{n}\}_{n\in\mathbb{N}} \in \mathcal{U} $ by induction. 

1. Select an arbitrary initial solution $ u_{0} \in \mathcal{U} $.  

2. Assume we have constructed a sequence $ u_{n} \in \mathcal{U} $ for $ n \geq 0 $.

3.  Define $ u_{n+1}  $ as the minimizer of the Tikhonov regularization functional:
\begin{align}
\label{eqFunctional}
J_{n} (\varphi) 
=
|\mathcal{P} \varphi - \mathbf{I} \mathcal{F}(u_{n}) |^{2}_{L^{2, w}_{\mathbb{F}}(\Omega; H^{1}(0,T; L^{2}(G)))}
+ \kappa |\varphi|_{\mathcal{H}^{2}_{\lambda, 0}}^{2}, 
\end{align}
where $ \varphi \in \mathcal{U} $ and $ \kappa \in (0,1) $ is the regularization parameter.

We have the following result on the existence and uniqueness of the solution to the minimization problem.

\begin{proposition}
For each fixed regularization parameter $ \kappa \in (0,1) $ and iteration index $ n \geq 0 $, the Tikhonov regularization functional $J_{n}$ defined in \cref{eqFunctional} admits a unique minimizer $u_{n+1}$ in the admissible set $\mathcal{U}$.
\end{proposition}
\begin{proof}
We begin by defining the weighted solution space:\vspace{-2mm}
\begin{align*}
\mathcal{U}_{0}^{w} \deq \Big\{ \varphi \in \mathcal{H}^{2} \;\Big|\;\mathcal{P} \varphi \in L^{2}_{\mathbb{F}}(\Omega; H^{1}(0,T; L^{2} (G))), 
\quad \varphi_{t}(0) = 0,
\quad  \varphi |_{\Sigma} = 0,  
\quad  \frac{\partial \varphi}{\partial \nu} \Big|_{\Sigma_{0}} = 0
\Big\}.
\end{align*}
For fixed $ \kappa \in (0,1) $, we endow $ \mathcal{U}_{0}^{w} $ with the inner product: \vspace{-2mm}
\begin{align}
\label{eqInnerProduct}
\langle p, q \rangle_{\mathcal{U}_{0}^{w}} 
\deq 
\langle \mathcal{P} p, \mathcal{P} q \rangle_{L^{2, w}_{\mathbb{F}}(\Omega; H^{1}(0,T; L^{2}(G)))}
+ \kappa \langle p, q \rangle_{\mathcal{H}^{2}_{\lambda, 0}},
\quad \forall\, p, q \in \mathcal{U}_{0}^{w}.
\end{align}
For simplicity of notations, we still denote by $ \mathcal{U}_{0}^{w} $ the completion of $ \mathcal{U}_{0}^{w} $ with respect to the norm $ | \cdot |_{\mathcal{U}_{0}^{w}} $.

Given a reference function $ R \in \mathcal{U}$, we consider the shifted variable $ v = \varphi - R $ and define the functional: \vspace{-2mm}
\begin{align*}
\overline{\mathcal{J}}_{n} (v) = 
|\mathcal{P} v + \mathcal{P} R - \mathbf{I} \mathcal{F}(u_{n}) |^{2}_{L^{2, w}_{\mathbb{F}}(\Omega; H^{1}(0,T; L^{2}(G)))}
+ \kappa |v+ R|_{\mathcal{H}^{2}_{\lambda, 0}}^{2}, 
\quad \forall\, v \in \mathcal{U}_{0}^{w}.
\end{align*}
The equivalence between minimizing $ \mathcal{J}_{n} $ 
and $ \overline{\mathcal{J}}_{n} $ 
is clear: $ v_{n+1} $ 
minimizes $ \overline{\mathcal{J}}_{n} $ if and only if $ u_{n+1} = v_{n+1} + R $ minimizes $ \mathcal{J}_{n} $.

The variational principle yields the Euler-Lagrange equation for the minimizer  $ v_{n+1} $:
\begin{align}
\label{eqEulerLagrange}
\langle \mathcal{P} v_{n+1} + \mathcal{P} R - \mathbf{I} \mathcal{F}(u_{n}), \mathcal{P} \rho \rangle_{L^{2, w}_{\mathbb{F}}(\Omega; H^{1}(0,T; L^{2}(G)))} 
+ \kappa \langle v_{n+1} + R, \rho \rangle_{\mathcal{H}^{2}_{\lambda, 0}} = 0,
\quad \forall\, \rho \in \mathcal{U}_{0}^{w}.
\end{align}

To verify that $ v_{n+1} $ is the minimizer of $ \overline{\mathcal{J}}_{n} $, consider for any $ \rho \in \mathcal{U}_{0}^{w} $:
\begin{align*}
& \overline{\mathcal{J}}_{n}(\rho) - \overline{\mathcal{J}}_{n}(v_{n+1}) 
\\
& =
| \mathcal{P} \rho + \mathcal{P} R - \mathbf{I} \mathcal{F}(u_{n}) |^{2}_{L^{2, w}_{\mathbb{F}}(\Omega;H^{1}(0,T;L^{2}(G)))} + \kappa | \rho |^{2}_{\mathcal{H}^{2}_{\lambda, 0}}
\\
& \quad 
- | \mathcal{P} v_{n+1} + \mathcal{P} R - \mathbf{I} \mathcal{F}(u_{n}) |^{2}_{L^{2, w}_{\mathbb{F}}(\Omega;H^{1}(0,T;L^{2}(G)))} 
- \kappa | v_{n+1} |^{2}_{\mathcal{H}^{2}_{\lambda, 0}}
\\
& = 
| \mathcal{P} v_{n+1} + \mathcal{P} R - \mathbf{I} \mathcal{F}(u_{n}) + \mathcal{P}(\rho- v_{n+1})|^{2}_{L^{2, w}_{\mathbb{F}}(\Omega;H^{1}(0,T;L^{2}(G)))} 
\\
& \quad 
+ \kappa | v_{n+1} + (\rho-v_{n+1}) |^{2}_{\mathcal{H}^{2}_{\lambda, 0}}
- | \mathcal{P} v_{n+1} + \mathcal{P} R - \mathbf{I} \mathcal{F}(u_{n}) |^{2}_{L^{2, w}_{\mathbb{F}}(\Omega;H^{1}(0,T;L^{2}(G)))} 
- \kappa | v_{n+1} |^{2}_{\mathcal{H}^{2}_{\lambda, 0}}
\\
& =
2 \langle \mathcal{P} v_{n+1} + \mathcal{P} R - \mathbf{I} \mathcal{F}(u_{n}), \mathcal{P} (\rho-v_{n+1})   \rangle _{L^{2, w}_{\mathbb{F}}(\Omega;H^{1}(0,T;L^{2}(G)))}
\\
& \quad 
+ 2 \kappa \langle v_{n+1}, \rho-v_{n+1} \rangle _{\mathcal{H}^{2}_{\lambda, 0}}
+ | \mathcal{P} (\rho - v_{n+1}) |^{2}_{L^{2, w}_{\mathbb{F}}(\Omega;H^{1}(0,T;L^{2}(G)))} 
+ \kappa | \rho-v_{n+1} |^{2}_{\mathcal{H}^{2}_{\lambda, 0}}
\\
& =
| \mathcal{P} (\rho - v_{n+1}) |^{2}_{L^{2, w}_{\mathbb{F}}(\Omega;H^{1}(0,T;L^{2}(G)))} 
+ \kappa | \rho-v_{n+1} |^{2}_{\mathcal{H}^{2}_{\lambda, 0}}
\\
& \geq
0,
\end{align*}
where we used \eqref{eqEulerLagrange} to eliminate the first two terms.  

Using \eqref{eqInnerProduct}, we can rewrite \eqref{eqEulerLagrange} as:
\begin{align}
\label{eqEulerLagrange2}
\langle v_{n+1}, \rho \rangle_{\mathcal{U}_{0}^{w}} 
= \langle \mathbf{I} \mathcal{F}(u_{n}) - \mathcal{P} R, \mathcal{P} \rho \rangle_{L^{2, w}_{\mathbb{F}}(\Omega; H^{1}(0,T; L^{2}(G)))},
\quad \forall\, \rho \in \mathcal{U}_{0}^{w}.
\end{align}
The right-hand side of \cref{eqEulerLagrange2} defines a bounded linear functional on $\mathcal{U}_{0}^{w}$  since
\begin{align*}
\langle \mathbf{I} \mathcal{F}(u_{n}) - \mathcal{P} R, \mathcal{P} \rho \rangle_{L^{2, w}_{\mathbb{F}}(\Omega; H^{1}(0,T; L^{2}(G)))}
\leq 
(|\mathcal{P} R|_{L^{2, w}_{\mathbb{F}}(\Omega; H^{1}(0,T; L^{2}(G)))} 
+ C  |u_{n}|_{\mathcal{H}^{1}_{\lambda, c_{0}}}
) | \rho|_{\mathcal{U}_{0}^{w}}.
\end{align*}
Thanks to Riesz representation theorem, there exists a unique $ w_{n+1} \in \mathcal{U}_{0}^{w} $ satisfying \eqref{eqEulerLagrange2}.

Assume that $ \widetilde{w}_{n+1} $ is another minimizer of $ \overline{\mathcal{J}}_{n} $. Then $ \widetilde{w}_{n+1} $ satisfies  \cref{eqEulerLagrange2}.
For all $ \rho \in \mathcal{U}_{0}^{w} $, we obtain \vspace{-3mm}
$$
\langle w_{n+1} - \widetilde{w}_{n+1}, \rho \rangle_{\mathcal{U}_{0}^{w}} 
= 0,
$$
which implies that $ w_{n+1} = \widetilde{w}_{n+1} $.
Hence, the minimizer of $ \overline{\mathcal{J}}_{n} $ is unique.
\end{proof}

Let $ u^{*} $
denote the exact solution to the semilinear stochastic hyperbolic system \eqref{eqSemlinearStochasticHyperbolic} with the exact boundary data $ u^{*} |_{\Sigma} = f^{*} $ and $ \dfrac{\partial u}{\partial \nu}  \bigg|_{\Sigma_{0}} = g^{*} $. 
Assume that there is a function $ \mathcal{E} \in \mathcal{H}^{2} $ satisfying\vspace{-2mm}
\begin{align}
\label{eqNoise}
|\mathcal{P} \mathcal{E} |_{L^{2}_{\mathbb{F}}(0,T; H^{1}(0,T; L^{2}(G)))} 
+ |\mathcal{E}|_{\mathcal{H}^{2}} \leq \delta,
\quad \mathcal{E}_{t}(0) = 0,
\quad 
f |_{\Sigma} = f^{*} + \mathcal{E}|_{\Sigma},
\quad 
g = g^{*} + \dfrac{\partial \mathcal{E}}{\partial \nu}  \bigg|_{\Sigma_{0}},\vspace{-2mm}
\end{align}
where $ \delta > 0 $ quantifies the noise level in the measurements.

\begin{remark}
Here, we assume that the boundary measurement data $f$ and $g$ can be used to construct a function $\mathcal{E}$, so that the error is considered only in terms of $\mathcal{E}$, as in \cite{Klibanov2021,Nguyen2022}. This is a key point that allows us to work with partial boundary measurements. However, our numerical experiments show that the method converges effectively even for more general noise, as in \cref{eqNoisyData}. 
At present, we do not know how to obtain convergence results for semilinear stochastic hyperbolic equations under the assumption $|f-f^{*}|+|g-g^{*}|<\delta$ in a suitable norm, as has been achieved for linear stochastic hyperbolic equations in \cite{Dou2024a}.
\end{remark}

We have the following convergence result.
\begin{theorem}
\label{thmConverge}
Let the regularization parameter $ \kappa > 0 $ be fixed, and $ c_{0} $ satisfy \cref{conPsi2}.
Fix $ \gamma \in (0,1) $. 
There exist positive constants $ C  $ and  $ \widetilde{\lambda}_{0} > C $ such that for all $ \lambda \geq \widetilde{\lambda}_{0} $, $ \delta < e^{-\lambda R_{1}^{2} / \gamma} $ and $ n \in \mathbb{N} $, we have\vspace{-4mm}
\begin{align*}
|u_{n+1} - u^{*}|_{\mathcal{H}^{1}_{\lambda, c_{0}}}^{2}
\leq 
\bigg(\frac{C}{\lambda}\bigg)^{n+1} | u_{0} - u^{*} |^{2}_{\mathcal{H}^{1}_{\lambda, c_{0}}} 
+ C (\delta^{2- 2 \gamma} + \kappa|u^{*}|^{2}_{\mathcal{H}^{2}_{\lambda, 0}})
,
\end{align*}
where $ u_{n+1} $ is the minimizer of the functional \cref{eqFunctional}.
\end{theorem}

\begin{proof} 
Since $ u_{n+1} $ minimizes \eqref{eqFunctional}, the Euler-Lagrange equation \eqref{eqEulerLagrange} gives 
\begin{align*}
\langle \mathcal{P} u_{n+1} - \mathbf{I} \mathcal{F}(u_{n}), \mathcal{P} \rho \rangle_{L^{2, w}_{\mathbb{F}}(\Omega; H^{1}(0,T; L^{2}(G)))}
+ \kappa \langle u_{n+1}, \rho \rangle_{\mathcal{H}^{2}_{\lambda, 0}} = 0, 
\quad \forall\, \rho \in \mathcal{U}_{0}^{w}.
\end{align*}
For the exact solution $ u^{*} $ 
satisfying $ \mathcal{P} u^{*} - \mathbf{I}\mathcal{F}(u^{*}) = 0 $, we have
\begin{align*}
\langle \mathcal{P} ^{*} - \mathbf{I} \mathcal{F}(u^{*}), \mathcal{P} \rho \rangle_{L^{2, w}_{\mathbb{F}}(\Omega; H^{1}(0,T; L^{2}(G)))}
+ \kappa \langle u^{*}, \rho \rangle_{\mathcal{H}^{2}_{\lambda, 0}} 
=  \kappa \langle u^{*}, \rho \rangle_{\mathcal{H}^{2}_{\lambda, 0}}, 
\quad \forall\, \rho \in \mathcal{U}_{0}^{w}.
\end{align*}
Subtracting the above two equations  yields that\vspace{-2mm}
\begin{equation}
\label{eqConvergence1}
\langle \mathcal{P} (u_{n+1}\! - u^{*})  -  ( \mathbf{I} \mathcal{F}(u_{n})  -  \mathbf{I} \mathcal{F}(u^{*})) , \mathcal{P} \rho \rangle_{L^{2, w}_{\mathbb{F}}(\Omega; H^{1}(0,T; L^{2}(G)))}
\!+  \kappa \langle u_{n+1} \!-\! u^{*}, \rho \rangle_{\mathcal{H}^{2}_{\lambda, 0}} 
= \!- \kappa \langle u^{*}, \rho \rangle_{\mathcal{H}^{2}_{\lambda, 0}}.\vspace{-2mm}
\end{equation}
Letting $ z = u_{n+1} -u^{*} - \mathcal{E} $, by \cref{eqNoise}, we get $ z \in \mathcal{U}_{0}^{w} $.
From \cref{eqConvergence1}, we arrive at
\begin{align*}
\langle \mathcal{P} (z + \mathcal{E}) - ( \mathbf{I} \mathcal{F}(u_{n}) - \mathbf{I} \mathcal{F}(u^{*})) , \mathcal{P} z  \rangle_{L^{2, w}_{\mathbb{F}}(\Omega; H^{1}(0,T; L^{2}(G)))}
+ \kappa \langle z + \mathcal{E}, z \rangle_{\mathcal{H}^{2}_{\lambda, 0}} 
= - \kappa \langle u^{*}, z \rangle_{\mathcal{H}^{2}_{\lambda, 0}}.
\end{align*}
Hence, we obtain \vspace{-4mm}
\begin{align}
\label{eqConvergence2} 
&|\mathcal{P} z|_{L^{2, w}_{\mathbb{F}}(\Omega; H^{1}(0,T; L^{2}(G)))}^{2}
+ \kappa |z|_{\mathcal{H}^{2}_{\lambda, 0}}^{2}\notag\\
& = 
\langle  ( \mathbf{I} \mathcal{F}(u_{n}) - \mathbf{I} \mathcal{F}(u^{*})) , \mathcal{P} z \rangle_{L^{2, w}_{\mathbb{F}}(\Omega; H^{1}(0,T; L^{2}(G)))}
-   \kappa \langle u^{*}, z \rangle_{\mathcal{H}^{2}_{\lambda, 0}} \\&\quad - \langle \mathcal{P}   \mathcal{E}  , \mathcal{P} z  \rangle_{L^{2, w}_{\mathbb{F}}(\Omega; H^{1}(0,T; L^{2}(G)))}
-   \kappa \langle \mathcal{E}, z \rangle_{\mathcal{H}^{2}_{\lambda, 0}}.\notag
\end{align}

Next, we estimate the right-hand side of \cref{eqConvergence2} one by one. 
From Cauchy-Schwarz inequality and \cref{conF}, for $ \varepsilon > 0 $, we get that\vspace{-2mm}
\begin{align}
\label{eqConvergence3} \notag
& \langle  ( \mathbf{I} \mathcal{F}(u_{n}) - \mathbf{I} \mathcal{F}(u^{*})) , \mathcal{P} z \rangle_{L^{2, w}_{\mathbb{F}}(\Omega; H^{1}(0,T; L^{2}(G)))}
\\ \notag
& \leq 
C(\varepsilon) | \mathbf{I} \mathcal{F}(u_{n}) - \mathbf{I} \mathcal{F}(u^{*}) |_{L^{2, w}_{\mathbb{F}}(\Omega; H^{1}(0,T; L^{2}(G)))}^{2} 
+ \varepsilon | \mathcal{P} z |_{L^{2, w}_{\mathbb{F}}(\Omega; H^{1}(0,T; L^{2}(G)))}^{2}
\\
& \leq 
C(\varepsilon) |u_{n} - u^{*}|_{\mathcal{H}^{1}_{\lambda, c_{0}}}^{2} 
+ \varepsilon | \mathcal{P} z |_{L^{2, w}_{\mathbb{F}}(\Omega; H^{1}(0,T; L^{2}(G)))}^{2}
,\vspace{-2mm}
\end{align}
\begin{align}
\label{eqConvergence4}  
-   \kappa \langle u^{*}, z \rangle_{\mathcal{H}^{2}_{\lambda, 0}}
-   \kappa \langle \mathcal{E}, z \rangle_{\mathcal{H}^{2}_{\lambda, 0}}
\leq 
\frac{\kappa}{\varepsilon} | u^{*} |_{\mathcal{H}^{2}_{\lambda, 0}}^{2}
+ \frac{\kappa}{\varepsilon} | \mathcal{E} |_{\mathcal{H}^{2}_{\lambda, 0}}^{2} 
+ 2 \varepsilon \kappa | z |_{\mathcal{H}^{2}_{\lambda, 0}}^{2}
,
\end{align}
and that
\begin{align}
\label{eqConvergence5}  
- \langle \mathcal{P}   \mathcal{E}  , \mathcal{P} z  \rangle_{L^{2, w}_{\mathbb{F}}(\Omega; H^{1}(0,T; L^{2}(G)))}
\leq 
C(\varepsilon) | \mathcal{P} \mathcal{E} |_{L^{2, w}_{\mathbb{F}}(\Omega; H^{1}(0,T; L^{2}(G)))}^{2} 
+ \varepsilon | \mathcal{P} z |_{L^{2, w}_{\mathbb{F}}(\Omega; H^{1}(0,T; L^{2}(G)))}^{2}.
\end{align}
Combining \cref{eqConvergence3,eqConvergence4,eqConvergence5,eqConvergence2}, choosing $ \varepsilon $ sufficiently small, we conclude that \vspace{-2mm}
\begin{align}
\label{eqConvergence6}
\notag
&    |\mathcal{P} z|_{L^{2, w}_{\mathbb{F}}(\Omega; H^{1}(0,T; L^{2}(G)))}^{2}
+ \kappa |z|_{\mathcal{H}^{2}_{\lambda, 0}}^{2}
\\
& \leq
C (
| u_{n} - u^{*} |_{\mathcal{H}^{1}_{\lambda, c_{0}}}^{2}  
+ \kappa |u^{*}|_{\mathcal{H}^{2}_{\lambda, 0}}^{2}
+ |\mathcal{P} \mathcal{E} |_{L^{2, w}_{\mathbb{F}}(\Omega; H^{1}(0,T; L^{2}(G)))}^{2}
+ |\mathcal{E}|_{\mathcal{H}^{2}_{\lambda, 0}}^{2}
).\vspace{-2mm}
\end{align}

Apply  \cref{thmCarlemanEstimate} to $ z $, for $ \lambda \geq \lambda_{0} $, we have\vspace{-2mm}
\begin{equation}
\label{eqConvergence7}
|\mathcal{P} z |_{L^{2, w}_{\mathbb{F}}(\Omega; H^{1}(0,T; L^{2}(G)))}^{2}
\geq 
\mathbb{E} \int_{Q} \theta^{2} (\lambda z_{t}^{2} + \lambda |\nabla z|^{2} + \lambda^{3} z^{2} ) d x d t 
- C \mathbb{E} \int_{G} \theta^{2} (\lambda |\nabla z|^{2} + \lambda^{3} |z|^{2}) \big|_{t = T} d x.\vspace{-2mm}
\end{equation}
From \cref{eqTheta} and trace theorem of Sobolev space, we obtain that\vspace{-2mm}
\begin{align}
\label{eqConvergence8} \notag
&\mathbb{E} \int_{G} \theta^{2} (\lambda |\nabla z|^{2} + \lambda^{3} |z|^{2}) \big|_{t = T} d x 
\\ \notag
& \leq 
\lambda^{3} \mathbb{E} \int_{G} e^{2 \lambda (\psi(x) - c_{0} T^{2})}  ( |\nabla z(T)|^{2} +  |z(T)|^{2})   d x 
\\ \notag
& \leq 
C \lambda^{5} e^{-2 c_{0} \lambda T^{2}} \mathbb{E} \int_{Q} e^{2 \lambda \psi(x)}  ( |z|^{2} +  |\nabla z|^{2} +  |z_{t}|^{2} + |\nabla^{2} z|^{2} + |\nabla z_{t}|^{2})   d x d t
\\
& = 
C \lambda^{5} e^{-2 c_{0} \lambda T^{2}} |z|_{\mathcal{H}^{2}_{\lambda, 0}}^{2}
.\vspace{-2mm}
\end{align}
Combining \cref{eqConvergence6,eqConvergence7,eqConvergence8}, for $ \lambda \geq \lambda_{0} $, we arrive at\vspace{-2mm}
\begin{align*}
&\mathbb{E} \int_{Q} \theta^{2} (\lambda z_{t}^{2} + \lambda |\nabla z|^{2} + \lambda^{3} z^{2} ) d x d t
+ \kappa |z|_{\mathcal{H}^{2}_{\lambda, 0}}^{2} 
- C \lambda^{5} e^{-2 c_{0} \lambda T^{2}} |z|_{\mathcal{H}^{2}_{\lambda, 0}}^{2}
\\
& \leq 
C \big(
| u_{n} - u^{*} |_{\mathcal{H}^{1}_{\lambda, c_{0}}}^{2}  
+ \kappa |u^{*}|_{\mathcal{H}^{2}_{\lambda, 0}}^{2}
+ |\mathcal{P} \mathcal{E} |_{L^{2, w}_{\mathbb{F}}(\Omega; H^{1}(0,T; L^{2}(G)))}^{2}
+ |\mathcal{E}|_{\mathcal{H}^{2}_{\lambda, 0}}^{2}
\big).
\end{align*}

By selecting $ \lambda_{1}  $ sufficiently large such that $ C\lambda^{5}  e^{-2 c_{0} \lambda_{1} T^{2}}  \leq \frac{\kappa}{2} $ and recalling that $ z = u_{n+1} -u^{*} - \mathcal{E} $, we deduce that for all $ \lambda \geq \max\{ \lambda_{0}, \lambda_{1} \} $, 
\begin{align*}
\lambda |u_{n+1} - u^{*} |_{\mathcal{H}^{1}_{\lambda, c_{0}}}^{2} 
\leq 
C (
| u_{n} - u^{*} |_{\mathcal{H}^{1}_{\lambda, c_{0}}}^{2}  
+ \kappa |u^{*}|_{\mathcal{H}^{2}_{\lambda, 0}}^{2}
+ |\mathcal{P} \mathcal{E} |_{L^{2, w}_{\mathbb{F}}(\Omega; H^{1}(0,T; L^{2}(G)))}^{2}
+ |\mathcal{E}|_{\mathcal{H}^{2}_{\lambda, 0}}^{2}).
\end{align*}

Next, setting $ \delta < e^{-\lambda R_{1}^{2} / \gamma} $ for $ x \in G $, we obtain the bound $
e^{2 \lambda \psi(x)} \leq e^{2\lambda R_{1}^{2}} \leq \delta^{- 2 \gamma}$. 
By \cref{eqNoise}, and choosing $ \widetilde{\lambda}_{0} \geq \max\{C, \lambda_{0}, \lambda_{1}\} $, we find that for all $ \lambda \geq \widetilde{\lambda}_{0}  $ and $ \delta < e^{-\lambda R_{1}^{2}/\gamma} $, the following inequality holds:\vspace{-2mm}
\begin{align}
\label{eqConvergence9}
|u_{n+1} - u^{*} |_{\mathcal{H}^{1}_{\lambda, c_{0}}}^{2} 
\leq 
\bigg(\frac{C}{\lambda}\bigg) | u_{n} - u^{*} |_{\mathcal{H}^{1}_{\lambda, c_{0}}}^{2}  
+ C (|u^{*}|_{\mathcal{H}^{2}_{\lambda, 0}}^{2}
+  \delta^{2- 2 \gamma} 
).
\end{align}
Finally, applying \cref{eqConvergence9} recursively completes the proof.
\end{proof}

\section{Numerical experiments for \texorpdfstring{Problem~\ref{prob1}}{Problem 1}}
\label{secNumerical}

In this section, we present some numerical tests to illustrate the performance of the proposed method.  
For simplicity, let the spatial domain be $ G = (0,1) \times (0, 1.5) $, the final time $ T =1 $ and $ b^{j k} = \delta_{jk} $ for $ j,k =1, 2 $. 
Consider the following semilinear stochastic hyperbolic equation:\vspace{-2mm}
\begin{align}\label{eqSemlinearStochasticHyperbolicNum}
\left\{
\begin{aligned}
& du_{t} - \Delta u dt=F(u, u_{t}, \nabla u) dt +  a u dW(t)&\mbox{ in }Q,\\
&   u_{t}(0) = 0 &\mbox{ in }G, \\
&   u(0) = u_{0} &\mbox{ in }G
,
\end{aligned}
\right.\vspace{-2mm}
\end{align}
where $ a \in L^{\infty}_{{\mathbb{F}}}(0,T;L^{\infty}(G)) $.

We solve the forward problem by employing the finite difference method in space and the explicit Euler-Maruyama method (see \cite{Lord2014,Kloeden1992}) in time.
Afterwards, we solve the inverse \cref{prob1} by iteratively minimizing the Tikhonov functional.

Before presenting the numerical results, we first introduce the weight function $ \theta $ in \cref{eqTheta} as follows:\vspace{-4mm}
\begin{align}
\label{eqThetaDis}
\theta(t,x,y) =
e^{\lambda (\psi(x) - 0.25 t^{2})}, \quad 
\psi(x,y) = (x - 0.5)^{2} + (y + 0.5 )^{2} - 4.175.
\end{align}
The weight function is chosen to satisfy \cref{conB,conF}. Accordingly, the observation boundary $\Gamma_{0}$ derived from \cref{eqGamma0} consists of: \vspace{-2mm}
$$
\Gamma_{0} = 
\{ (1,y) \mid y \in [0, 1.5] \}
\cup \{ (x, 1.5) \mid x \in [0,1] \}
\cup \{ (0, y) | y \in [0, 1.5] \}.\vspace{-2mm}
$$
We shall employ the lateral Cauchy data  $ u|_{(0,T)\times \Gamma} $ and $ \frac{\partial u}{\partial \nu}\Big|_{\Sigma_{0}} $
to reconstruct the unknown initial source function $ u_0 $.

\begin{remark}
\label{rkLessBoundary}

It is worth noting that, unlike in the deterministic case of \cite{Nguyen2022}, solving \cref{prob1} does not require the boundary data for $ \frac{\partial u}{\partial \nu} $ on the lower boundary $\{ (x,0) \mid x \in [0, 1] \}$, which is typically more difficult to obtain.
We also employ a rectangular domain instead of a square one in \cite{Nguyen2022}, which usually poses difficulties in reconstructing the initial condition.

\end{remark}

Now we present the computational algorithm employed in our numerical experiments.  Let
$J_{n, \ell}(u; u_n)$ denote the Tikhonov functional defined in \cref{eqFunctional}, where the subscript $\ell$ indicates the dependence on sample paths $\omega^{\ell}$ through the It\^o integration term in operator $\mathcal{P}$. 
The inverse problem is solved through the following iterative optimization procedure (Algorithm \ref{algInverseProblem}), where $ \mathfrak{g}_{t}^{2} $ represents the element-wise square of $ \mathfrak{g}_{t} $: 
\begin{algorithm}
\caption{The algorithm for solving the inverse problem}
\label{algInverseProblem}
\begin{algorithmic}[1]
\STATE  Choose maximum iteration $N_{I}$ and the number of sample paths $ N_{S} $.
\FOR{ $ \ell = 1 $ to $ N_{S} $}
    \STATE  Choose the initial guess $ u_{\ell}^{0} \in \mathcal{U} $.
    \FOR{ $ n=1 $ to $ N_{I} $ }
    \STATE Choose the maximum number of update steps $N_U$, the learning rate $\alpha$, and the parameters $\beta_{1}$, $\beta_{2}$, and $\varepsilon$.
    \STATE Initialize $u_{(0)} = u_{l}^{n-1}$, $v_{0} = 0$, $m_{0} = 0$.
    \FOR{ $ t=1 $ to $ N_{U} $}
        \STATE Set $ \mathfrak{g}_{t} = \nabla_{u} J_{n,\ell}(u_{(t-1)}; u_{\ell}^{n-1}) $.
        \STATE Set $ m_{t} = \beta_{1} m_{t-1} + (1-\beta_{1}) \mathfrak{g}_{t} $  and $ v_{t} = \beta_{2} v_{t-1} + (1-\beta_{2}) \mathfrak{g}_{t}^{2} $.
        \STATE Set $ \hat{m}_{t} = m_{t} / (1 - \beta_{1}^{t}) $ and $ \hat{v}_{t} = v_{t} / (1 - \beta_{2}^{t}) $.
        \STATE Set $ u_{(t)} = u_{(t-1)} - \alpha \hat{m}_{t} / (\sqrt{\hat{v}_{t}} + \epsilon) $.
    \ENDFOR
    \STATE Set $u_{\ell}^{n} = u_{(N_{U})}$ as the minimizer of the functional $J_{n,\ell}(u; u_{\ell}^{n-1})$.
\ENDFOR
\ENDFOR
\STATE Set the compute solution $ u_{c} = \frac{1}{N_{S}} \sum\limits_{\ell=1}^{N_{S}} u_{\ell}^{N_{I}} $.
\STATE \textbf{Output:} The reconstructed initial source function $u_{0} = u_{c}(0, x, y)$ for $ (x,y) \in G $.
\end{algorithmic}
\end{algorithm}

The core challenge in \cref{algInverseProblem} lies in minimizing the functional $ J_{n,\ell}(u; u_{\ell}^{n-1}) $, which is carried out in steps 5--13. This minimization is complicated since even when $\ell$ is fixed, the lateral Cauchy data 
$u\big|_{(0,T)\times\Gamma}$ and 
$\frac{\partial u}{\partial \nu}\big|_{(0,T)\times\Gamma_0}$ 
are stochastic processes corresponding to the solution of the stochastic hyperbolic equation. 
Consequently, if one employs the conjugate gradient method (CGM for short), which is commonly used in deterministic settings to compute the gradient (e.g., \cite{Klibanov2016}), the resulting gradient will inherit this randomness, potentially resulting in the proper convergence (see \cref{fig:compareSGD_Adam}).
Moreover, the inherent randomness makes it difficult to reformulate the problem into solving a linear equation subject to certain boundary and initial constraints, as is done using the quasi-reversibility method in \cite{Nguyen2022,Nguyen2019}. 
Therefore, we will use the Adam optimizer \cite{Kingma2014} to minimize the functional $ J_{n,\ell}(u; u_{\ell}^{n-1}) $.  

The Adam algorithm is specifically designed to address stochastic optimization problems. Its key innovation lies in the incorporation of a momentum mechanism: 
\begin{align*}
m_{t} = \beta_{1} m_{t-1} + (1-\beta_{1}) \mathfrak{g}_{t}.
\end{align*}
This approach offers two significant benefits:

1. Noise Reduction: By computing an exponentially weighted average of past gradients, the algorithm effectively smooths out stochastic fluctuations in the gradient estimates.

2. Stable Convergence: The momentum term helps maintain consistent update directions, leading to more reliable convergence behavior compared to standard stochastic gradient methods.

\begin{figure}
\centering
\includegraphics[width=0.8\textwidth]{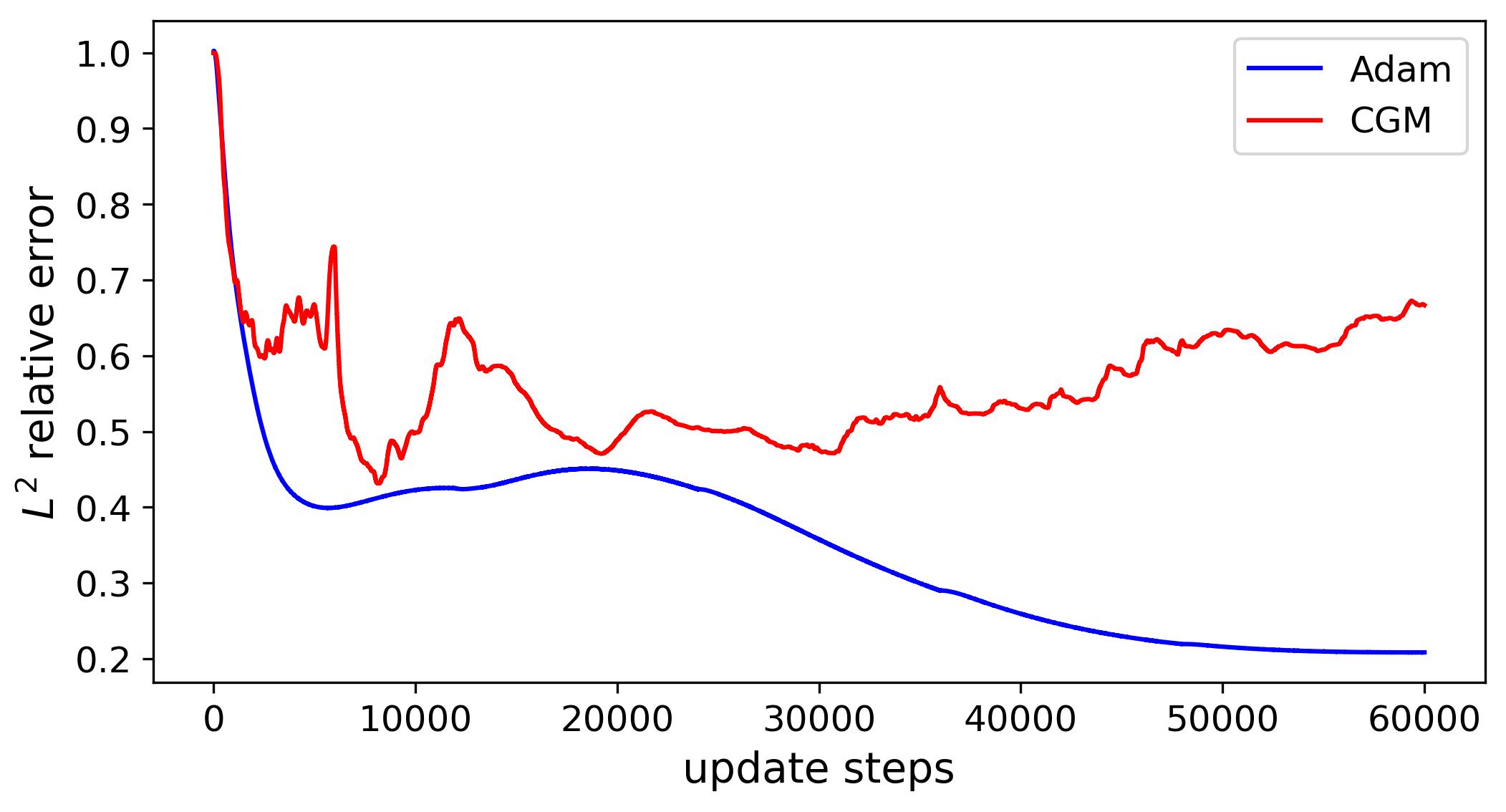}
\caption{We adopt the setting of \cref{ex4} with $N_{S}=1$. The figure compares the performance of the Adam optimizer and the CGM in solving the inverse problem, demonstrating that Adam achieves faster convergence and more stable optimization when handling stochastic gradients and noisy data.}
\label{fig:compareSGD_Adam}
\end{figure}

On the other hand, our computations reveal a significant disparity in the magnitudes of different gradient components, due to the exponential form of the weight function $\theta$ defined in \cref{eqThetaDis}. 
This issue reduces the convergence rate of the CGM (see \cref{fig:compareSGD_Adam}).
However, Adam mitigates the problem by adjusting the learning rate adaptively based on the historical gradient information for each parameter.
More precisely, the learning rate is given by
\begin{align*}
\gamma_{t} = \frac{\gamma}{\sqrt{\hat{v}_{t}} + \varepsilon},
\end{align*}
where $ v_{t} = \beta_{2} v_{t-1} + (1-\beta_{2}) \mathfrak{g}_{t}^{2}$ and $\hat{v}_{t} = v_{t} / (1 - \beta_{2}^{t}) $.

The next challenging task is to compute the gradient 
$ \nabla_{u} J_{n,\ell}(u_{(t-1)}; u_{\ell}^{n-1}) $. 
For stochastic problems, using the adjoint method typically requires solving an adjoint system, that is, a backward stochastic differential equation (see, e.g., \cite{Dou2022}).
Numerically, the evaluation of the correction term in a backward stochastic differential equation is very difficult because, to date, there is no effective method to compute the conditional expectation of a general random variable \cite{Lue2022}. 
To overcome this issue, we employ the technique of automatic differentiation to compute the gradient.

One of the fundamental reasons behind the significant advances in machine learning is the underlying computational infrastructure centered on automatic differentiation. The basic idea of automatic differentiation is to decompose a function into a series of simple operations and then use the chain rule to compute derivatives. 
Automatic differentiation is an efficient method for evaluating the gradient of the complicated functional, with a computational cost independent of the number of parameters. 
Modern open-source machine learning libraries such as PyTorch \cite{Paszke2017} have already incorporated this method, and it has been successfully applied in areas such as computational fluid dynamics \cite{Alhashim2025}. 

In this paper, the functional $J_{n,\ell}(u; u_{\ell}^{n-1})$ has intrinsic randomness (see \cref{rkDiscreteFunctional}). Furthermore, its parameter dimension is extremely high (approximately 80,000 parameters in our numerical examples, see \cref{eqTensorDis}). Hence, automatic differentiation is employed to efficiently compute the gradient of $J_{n,\ell}(u; u_{\ell}^{n-1})$ for each sample path. This approach avoids the need for an explicit derivative calculation or the solution of a backward stochastic partial differential equation. 
Our subsequent numerical experiments demonstrate the efficiency and accuracy of this approach.

\subsection{Generating the simulated data}

Introduce the uniform mesh in $ (0,T) \times G  $ as 
\begin{align*}
\mathfrak{M} = \{ (t_{k}, x_{i}, y_{j}) \mid 
t_{k} = k \tau, \; x_{i} = i h_{x}, \; y_{j} = j h_{y}, \;
k \in \llbracket 0, M \rrbracket, \; i \in \llbracket 0, N_{x} \rrbracket ,\;
j \in \llbracket 0, N_{y} \rrbracket
\},
\end{align*}
where $ h_{x} = 1 / N_{x} $, $ h_{y} = 3/ (2 N_{y}) $ and $ \tau = 1 / M $.
Here, we use the notation $ \llbracket a, b \rrbracket = [a,b] \cap \mathbb{N} $.
We choose $ M = 65 $, $ N_{x}=32 $ and $ N_{y} = 48 $. 
These values are chosen such that Courant--Friedrichs--Lewy condition $ \tau ( 1/ h_{x}^{2} + 1/ h_{y}^{2} )^{1/2} \simeq 0.862 <1 $ is satisfied.
Thus, the computed solution obtained by our scheme is stable.

Let $ u_{k,i,j} = u(t_{k}, x_{i}, y_{j}) $ for $ k \in \llbracket 0, M \rrbracket, \; i \in \llbracket 0, N_{x} \rrbracket ,\;
j \in \llbracket 0, N_{y} \rrbracket $.
The equation \cref{eqSemlinearStochasticHyperbolicNum} is solved by centred difference
approximation for the Laplacian operator and explicit Euler-Maruyama method for the time.
For $ k \in \llbracket 1, M-1 \rrbracket, \; 
i \in \llbracket 1, N_{x}-1 \rrbracket ,\;
j \in \llbracket 1, N_{y}-1 \rrbracket  $, we solve the following equation:
\begin{align}
\label{eqDisFor} 
[d u_{t}]_{k,i,j}
& =
[u_{xx}]_{k,i,j} \tau
\!+\!  [u_{yy}]_{k,i,j} \tau
\!+\! F(t_{k}, x_{i}, y_{j}, u_{k, i, j}, [u_{t}]_{k, i, j}, [\nabla u]_{k, i, j}) \tau 
\!+\!a(t_{k}, x_{i}, y_{j}) u_{k, i, j}  \Delta W_{k}
,
\end{align}
where the increments $ \Delta W_{k} = W_{k+1} - W_{k} \sim N(0, \sqrt{\tau}) $, and
\begin{align}
\label{equtDis} \notag
[d u_{t}]_{k,i,j} & = \frac{u_{k+1,i,j} - 2 u_{k,i,j} + u_{k-1,i,j}}{\tau } ,
&&[u_{xx}]_{k,i,j}   = \frac{u_{k, i+1, j} - 2 u_{k, i, j} + u_{k, i-1, j}}{h_{x}^{2}}
\\ \notag
[u_{yy}]_{k,i,j} & = \frac{u_{k, i, j+1} - 2 u_{k, i, j} + u_{k, i, j-1}}{h_{y}^{2}}.
&&[\nabla u]_{k, i, j}  = \left( \frac{u_{k, i+1, j} - u_{k, i, j}}{ h_{x}}, \frac{u_{k, i, j+1} - u_{k, i, j}}{ h_{y}} \right)^{T},
\\
[u_{t}]_{k, i, j} &= \frac{u_{k+1, i, j} - u_{k,i,j}}{ \tau}. 
\end{align}
The initial value condition $ u(0) = u_{0} $ and $ u_{t}(0) = 0 $ is given by 
\begin{align}
\label{eqInitCondi}
u_{0,i,j} = u_{0}(x_{i}, y_{j}), \quad 
u_{1,i,j} = u_{0,i,j}, \quad 
i \in \llbracket 0, N_{x} \rrbracket ,\;
j \in \llbracket 0, N_{y} \rrbracket.
\end{align}
The boundary condition $ u|_{(0,T) \times \Gamma} = f^{*} $ is given by
\begin{align}
\label{eqBoundaryCon1}
u_{k, i, j} = \left\{
\begin{aligned}
f^{*} (t_{k}, 1, y_{j}), \quad & k \in \llbracket 0, M  \rrbracket , \; i = N_{x}, \; j \in \llbracket 0, N_{y} \rrbracket,\\
f^{*} (t_{k}, 0, y_{j}), \quad & k \in \llbracket 0, M  \rrbracket , \; i = 0, \; j \in \llbracket 0, N_{y} \rrbracket,\\
f^{*} (t_{k}, x_{i}, 0), \quad & k \in \llbracket 0, M  \rrbracket , \; i \in \llbracket 0, N_{x} \rrbracket ,\; j = 0,\\
f^{*} (t_{k}, x_{i}, 3/2), \quad & k \in \llbracket 0, M  \rrbracket , \; i \in \llbracket 0, N_{x} \rrbracket ,\; j = N_{y}.
\end{aligned}
\right.
\end{align}
After solving \cref{eqDisFor}, the boundary data for $\frac{\partial u}{\partial \nu} = g^{*}$ on $ (0,T) \times \Gamma_{0} $ is computed using 
\begin{align}
\label{eqBoundaryCon2}
g^{*} (t_{k}, x_{i}, y_{j}) = \left\{
\begin{aligned}
\frac{u_{k, i, j} - u_{k, i-1, j}}{ h_{x}}, \quad & k \in \llbracket 0, M  \rrbracket , \; i = N_{x}, \; j \in \llbracket 0, N_{y} \rrbracket,\\
\frac{u_{k, i, j} - u_{k, i+1, j}}{ h_{x}}, \quad & k \in \llbracket 0, M  \rrbracket , \; i = 0, \; j \in \llbracket 0, N_{y} \rrbracket,\\
\frac{u_{k, i, j} - u_{k, i, j-1}}{ h_{y}}, \quad & k \in \llbracket 0, M  \rrbracket , \; i \in \llbracket 0, N_{x} \rrbracket ,\; j = N_{y}
.
\end{aligned}
\right.
\end{align}
Recalling \cref{rkLessBoundary}, we do not need to compute the boundary data for $ \frac{\partial u}{\partial \nu} $ on the lower boundary $\{ (x,0) \mid x \in [0, 1] \}$.

\begin{remark}
\label{rkForwardBoundary}
Note that $f^{*}$ and $g^{*}$ actually depend on the sample point $\omega$. In the sequel, a total of $8$ sample paths for the forward equations were generated for each test.
\end{remark}

\subsection{Specifying the Tikhonov functional}

To solve the inverse problem, we employ the Tikhonov regularization approach using the functional defined in \cref{eqFunctional}. The numerical implementation requires the following key components:

1. Discretization Scheme:
The optimization variable is discretized as a three-dimensional tensor:
\begin{align}
\label{eqTensor}
(u_{k,i,j})_{k\in \llbracket 0,M\rrbracket,; i\in \llbracket 0,N_x\rrbracket,; j\in \llbracket 0,N_y\rrbracket},
\end{align}
representing the solution across temporal and spatial dimensions.

2. Noise Model:
We introduce 10\% multiplicative noise to the lateral Cauchy data:
\begin{align}
\label{eqNoisyData} \notag
f(t,x,y) &= f^{*}(t,x,y)(1 + \delta\xi(t,x,y)), \\
g(t,x,y) &= g^{*}(t,x,y)(1 + \delta\xi(t,x,y)),
\end{align}
where $\xi(t,x,y)$  is a random variable uniformly distributed in $[-1,1]$, which models uniformly distributed measurement noise. 

3. Discrete Feasible Set:
The constrained solution space is defined as:
\begin{align*}
\mathcal{U}_{D} =\{
(u_{k,i,j})_{k \in \llbracket 0, M \rrbracket,\; i \in \llbracket 0, N_{x} \rrbracket ,\; j \in \llbracket 0, N_{y} \rrbracket}
\mid 
(u_{k,i,j})  \text{ fulfills \cref{eqInitCondi,eqBoundaryCon1,eqBoundaryCon2}  with $ f^{*}, g^{*} $ replaced by $ f, g $}
\}.
\end{align*}

\begin{remark}
Note that $f^{*}$ and $g^{*}$ depend on the sample point $\omega$ of the forward equation \cref{eqDisFor}. In step 8 of Algorithm \ref{algInverseProblem}, inspired by stochastic gradient descent \cite{Robbins1951}, we compute the gradient corresponding to the Cauchy data $f^{*}$ and $g^{*}$ for each sample $\omega$ of \cref{eqDisFor} in turn. For brevity, we do not distinguish between the $f^{*}$ and $g^{*}$ generated by different sample paths here.
\end{remark}

To compute the Tikhonov functional \cref{eqFunctional}, we first discretize the operator $\mathcal{P}$ appearing in \cref{eqFunctional}. Observe that if $\varphi$ satisfies
\begin{align*}
d\varphi_{t} - \Delta \varphi  dt = \Phi  dt + a\varphi  dW(t),
\end{align*}
then it follows that
\begin{align*}
& |\mathcal{P}\varphi - \mathbf{I}\mathcal{F}(u_{n})|^{2}_{L^{2,w}_{\mathbb{F}}(\Omega;H^{1}(0,T;L^{2}(G)))} 
\\
& = \mathbb{E}\int_{0}^{T}\int_{G}\theta^{2} \Big[\int_{0}^{t}\big(\Phi-\mathcal{F}(u_{n})\big)ds\Big]^{2}dxdt 
+ \mathbb{E}\int_{0}^{T}\int_{G}\theta^{2}|\Phi-\mathcal{F}(u_{n})|^{2}dxdt.
\end{align*}
Motivated by this observation, we define a new Tikhonov functional
\begin{align}
\label{eqNewFunctional}
\widehat{J}_{n}(\varphi;u_{n})
= \mathbb{E}\int_{0}^{T}\int_{G}\theta^{2}|\Phi-\mathcal{F}(u_{n})|^{2}dxdt  
+ \kappa |\varphi|_{\mathcal{H}^{2}_{\lambda,0}}^{2}.
\end{align}
By the arguments in the proof of \cref{thmConverge}, the minimizer of $\widehat{J}_{n}(\varphi;u_{n})$ also satisfies the convergence result in \cref{thmConverge}. 
Hence, in the sequel, we consider the discrete form of \cref{eqNewFunctional} for $ u = (u_{k,i,j}) \in \mathcal{U}_{D} $:
\begin{align}
\label{eqDiscreteFunctional} \notag
J_{n,\ell} (u; u_{\ell}^{n-1}) 
&=
\sum_{k=1}^{M-1} \sum_{i=1}^{N_{x}-1} \sum_{j=1}^{N_{y} -1} 
\Big\{ \theta^{2}_{k,i,j} \big(\Phi_{k,i,j} - [\mathcal{F}(u_{\ell}^{n-1})]_{k,i,j}\big)^{2} \tau h_{x} h_{y}
\\ \notag
& \quad  \quad \quad \quad \quad \quad \quad 
+ [\theta_{0}]^{2}_{k,i,j} \big(
u_{k,i,j}^{2} + [u_{t}]_{k,i,j}^{2} + [\nabla u]^{2}_{k,i,j} + [u_{xx}]^{2}_{k,i,j}
+ [u_{yy}]^{2}_{k,i,j}
\\
& \quad \quad \quad \quad \quad \quad \quad \quad \quad \quad \quad 
+[u_{xy}]^{2}_{k,i,j}
+ [u_{tx}]^{2}_{k,i,j} 
+ [u_{ty}]^{2}_{k,i,j} 
\big)
\Big\}
\tau h_{x} h_{y}
,
\end{align}
where 
\begin{align*}
&
\theta_{k,i,j} = e^{  \lambda (\psi(x_{i}, y_{j})- c_{0} t_{k}^{2})},  \quad \quad 
[\theta_{0} ]_{k,i,j} = e^{  \lambda \psi(x_{i}, y_{j})} ,
\\
&
\Phi_{k,i,j} = \{[du_{t}]_{k,i,j} - ([u_{xx}]_{k,i,j} + [u_{yy}]_{k,i,j}) \tau - a_{k,i,j} u_{k,i,j} \Delta W_{k}   \} /\tau,
\\
&
[\mathcal{F}(u_{\ell}^{n-1})]_{k,i,j} = F(t_{k}, x_{i}, y_{j}, [u_{\ell}^{n-1}]_{k, i, j}, [\partial_{t}u_{\ell}^{n-1}]_{k, i, j}, [\nabla u_{\ell}^{n-1}]_{k, i, j}),
\\
&
[u_{xy}]_{k,i,j} = \frac{u_{k, i+1, j+1} - u_{k, i+1, j} - u_{k, i, j+1} + u_{k, i, j}}{ h_{x} h_{y}},
\\
&
[u_{tx}]_{k,i,j} = \frac{u_{k+1, i+1, j} - u_{k+1, i, j} - u_{k, i+1, j} + u_{k, i, j}}{ h_{x} \tau},
\\
&
[u_{ty}]_{k,i,j} = \frac{u_{k+1, i, j+1} - u_{k+1, i, j} - u_{k, i, j+1} + u_{k, i, j}}{ h_{y} \tau}.
\end{align*}
Other notations are defined in \cref{equtDis}.

To keep the initial value condition and boundary value conditions throughout all update steps, we minimize the functional \cref{eqDiscreteFunctional} with respect to the tensor
\begin{align}
\label{eqTensorDis}
(u_{k,i,j})_{k\in \llbracket 1,\,M \rrbracket,\; i\in \llbracket 2,\,N_x-2 \rrbracket,\; j\in \llbracket 2,\, N_y-2 \rrbracket}.
\end{align}
We employ automatic differentiation techniques \cite{Paszke2017} to directly compute the gradient of the discrete functional \cref{eqDiscreteFunctional} with respect to the tensor \cref{eqTensorDis}. 
For $N_{U} = 12000$, the minimization of \cref{eqDiscreteFunctional} is completed in only 15 seconds on an Apple M3 Pro.

\begin{figure}
\centering
\includegraphics[width=0.8\textwidth]{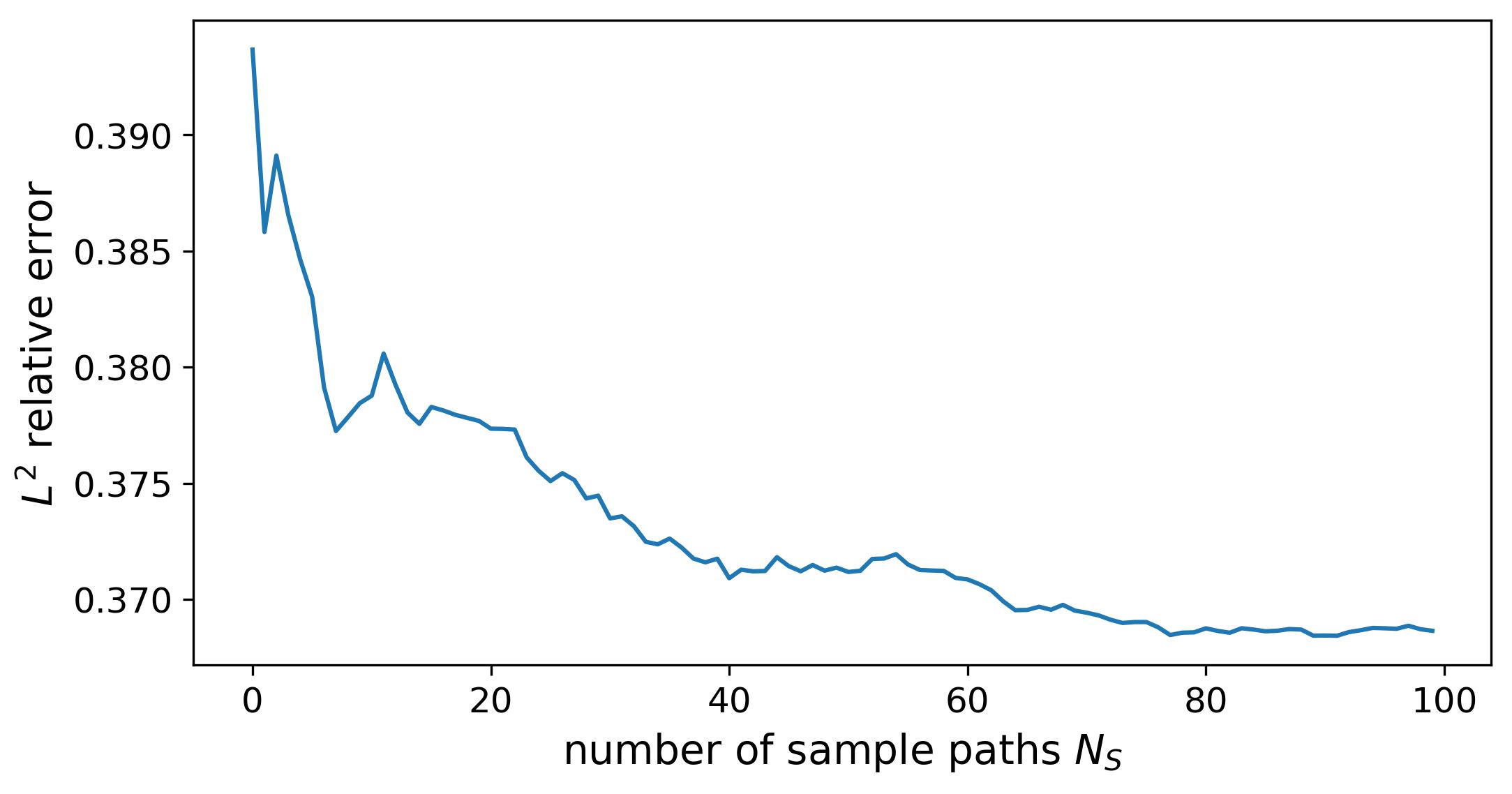}
\caption{We adopt the setting of \cref{ex1} with $N_{U}=3000$, $N_{I}=3$, and $N_{S}=100$. The figure shows how the $L^2$ relative error 
$\frac{|u_{0} - u_{0}^{*}|_{L^{2}(G)}}{|u_{0}^{*}|_{L^{2}(G)}}$ decreases with the number of sample paths $N_{S}$. In subsequent numerical tests, we set $N_{S}=30$.}
\label{fig:L2_per_samples}
\end{figure}

\begin{remark}
\label{rkDiscreteFunctional}
Note that when computing the discrete functional \cref{eqDiscreteFunctional}, one must sample the Wiener increments $\Delta W_{k}$, which inherently introduces randomness into the functional.
In \cref{algInverseProblem}, we perform $N_{S}=30$ samples (see \cref{fig:L2_per_samples}).
\end{remark}

\subsection{Numerical examples}

For all numerical experiments presented in this work, we employ the following set of optimized hyperparameter:
$ \lambda = 0.2 $, $ c_{0} = 0.25 $,  $ \kappa = 10^{-4}$, $ N_{S} = 30 $, $ N_{I} = 5 $, $ N_{U} = 12000 $, $ \alpha = 0.01 $, $ \beta_{1} = 0.9 $, $ \beta_{2} = 0.999 $ and $ \varepsilon = 10^{-8} $.

For all numerical experiments, we initialize the algorithm with a trivial initial guess $u_{\ell}^{0} \equiv 0$. 
Our method does not require a specially constructed initial guess that satisfies the boundary conditions. This is in contrast to the approach in \cite{Nguyen2022}, where the initial guess is obtained by solving a regularization problem. Nevertheless, our method converges to the true solution even when starting from zero.

We use the $ L^{2} $ relative error \vspace{-3mm}
\begin{align*}
\frac{\mathbb{E}|u_{0}-u_{0}^{*}|_{L^{2}(G)}}{\mathbb{E}|u_{0}^{*}|_{L^{2}(G)}} 
\end{align*}
to measure the accuracy of the computed solution, where $u_{0}^{*}$ is the true source function and $u_{0}$ is the computed source function.

\begin{figure}[!htb]
\centering

\begin{subfigure}[t]{0.3\textwidth}
\centering
\includegraphics[width=\linewidth]{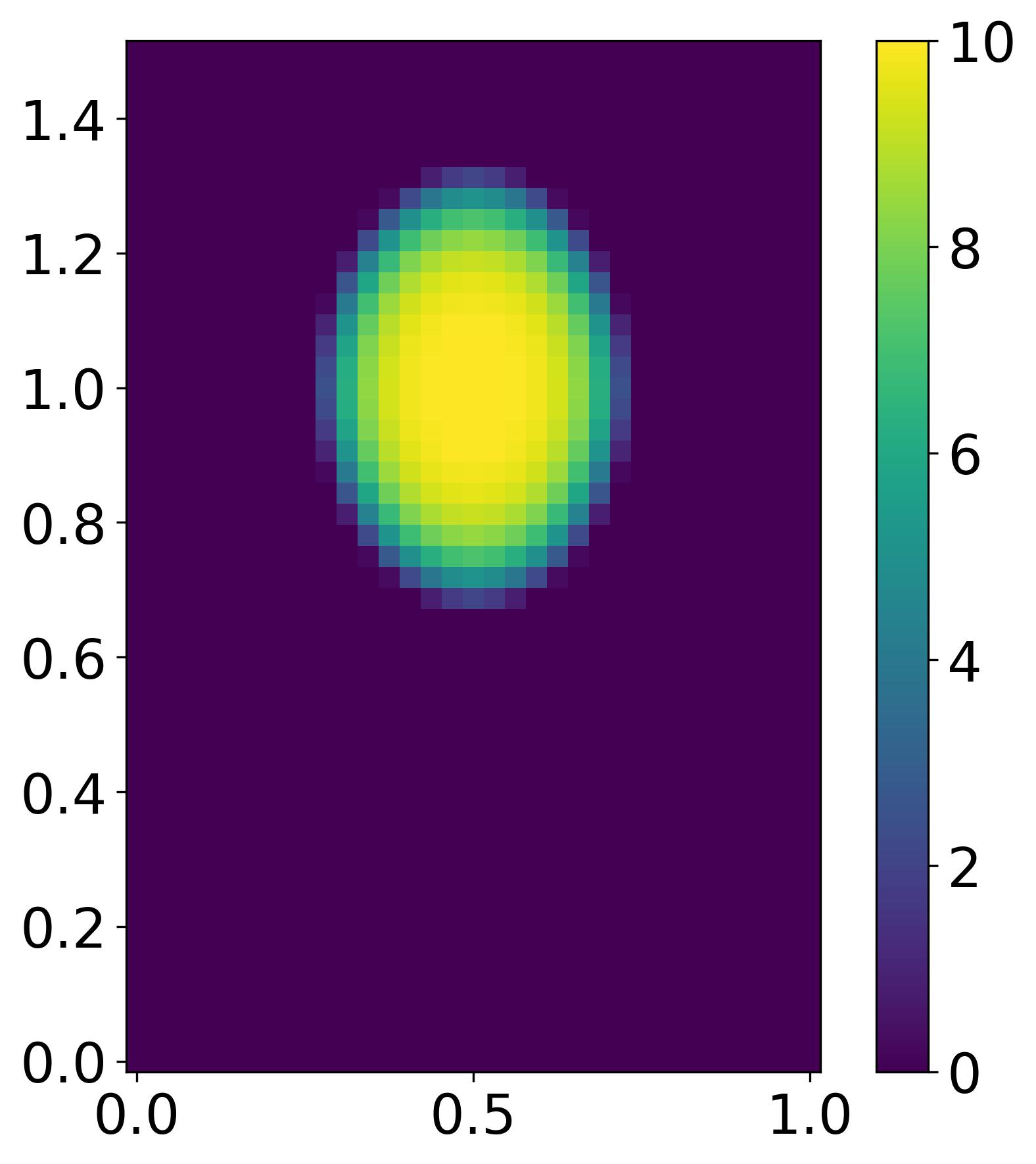}
\caption{The true source function $u_{0}^{*}$.}
\label{fig.exp1a}
\end{subfigure}%
\hfill
\begin{subfigure}[t]{0.3\textwidth}
\centering
\includegraphics[width=\linewidth]{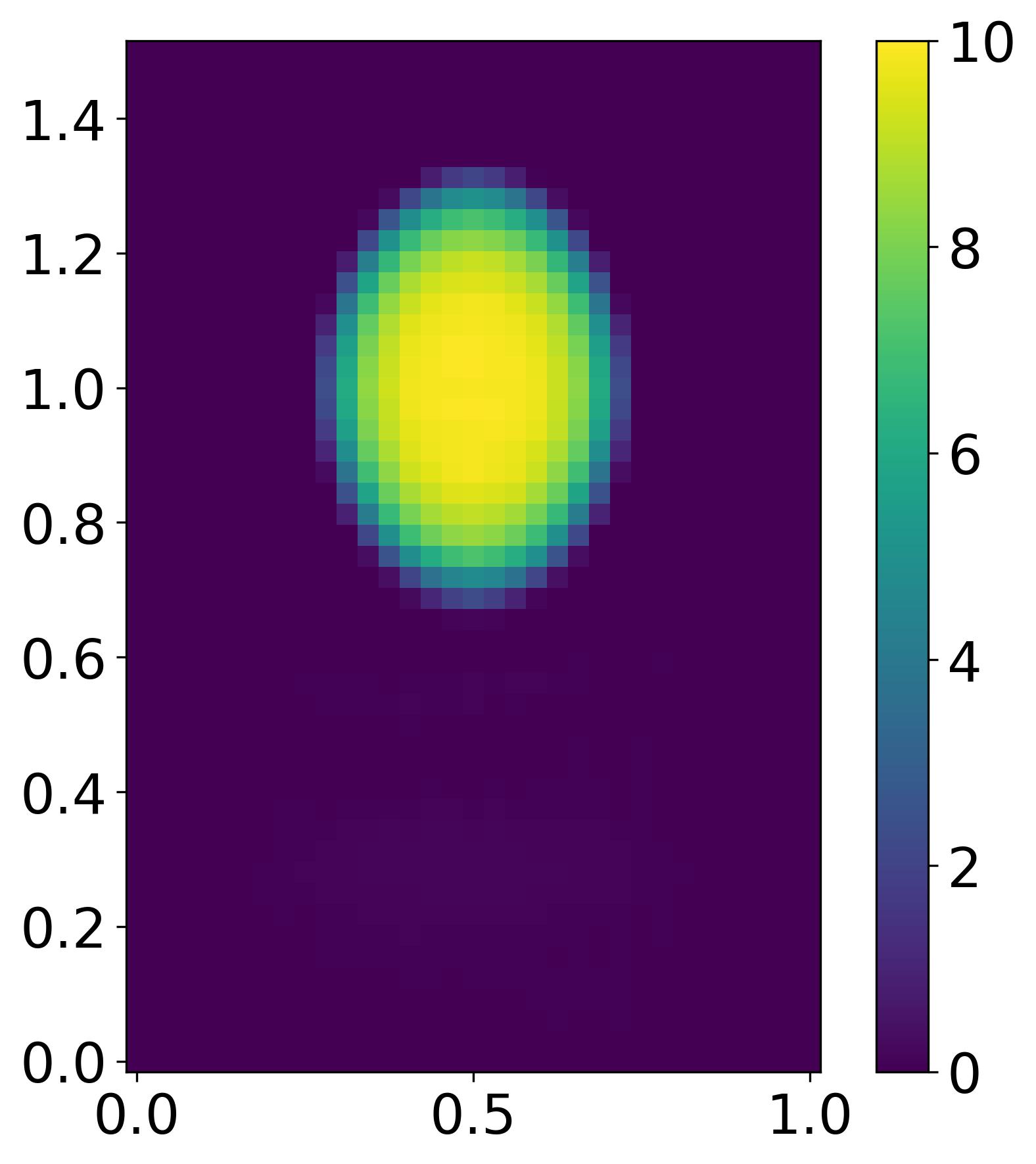}
\caption{The computed source function $u_{0}$.}
\label{fig.exp1b}
\end{subfigure}%
\hfill
\begin{subfigure}[t]{0.3\textwidth}
\centering
\includegraphics[width=1.03 \linewidth]{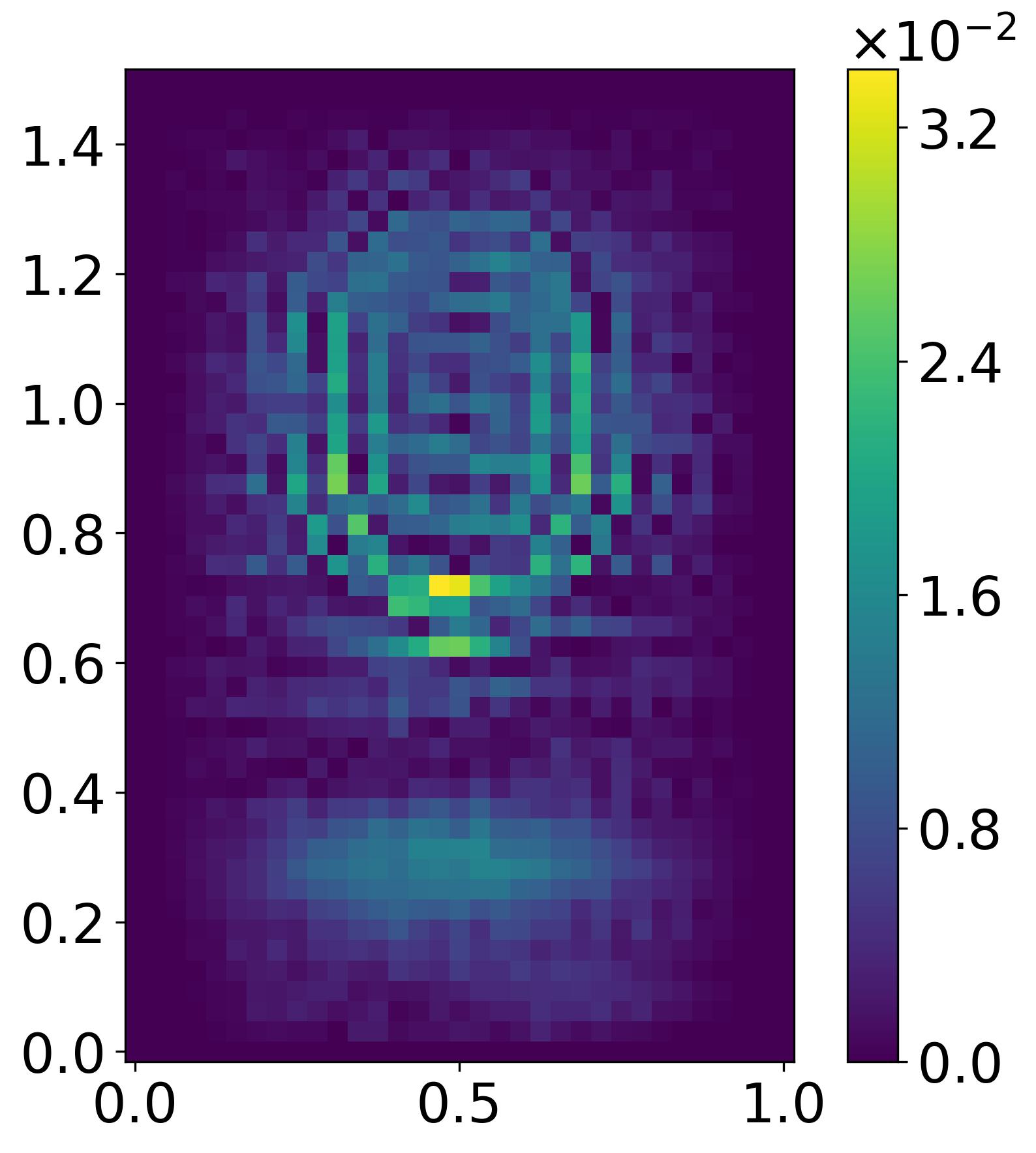}
\caption{The relative difference $ \frac{|u_{0} - u_{0}^{*}|}{|u_{0}^{*}|_{L^{\infty}(G)}} $.}
\label{fig.exp1c}
\end{subfigure}

\vspace{0.5cm} 

\begin{subfigure}[t]{0.47\textwidth}
\centering
\includegraphics[width=\linewidth]{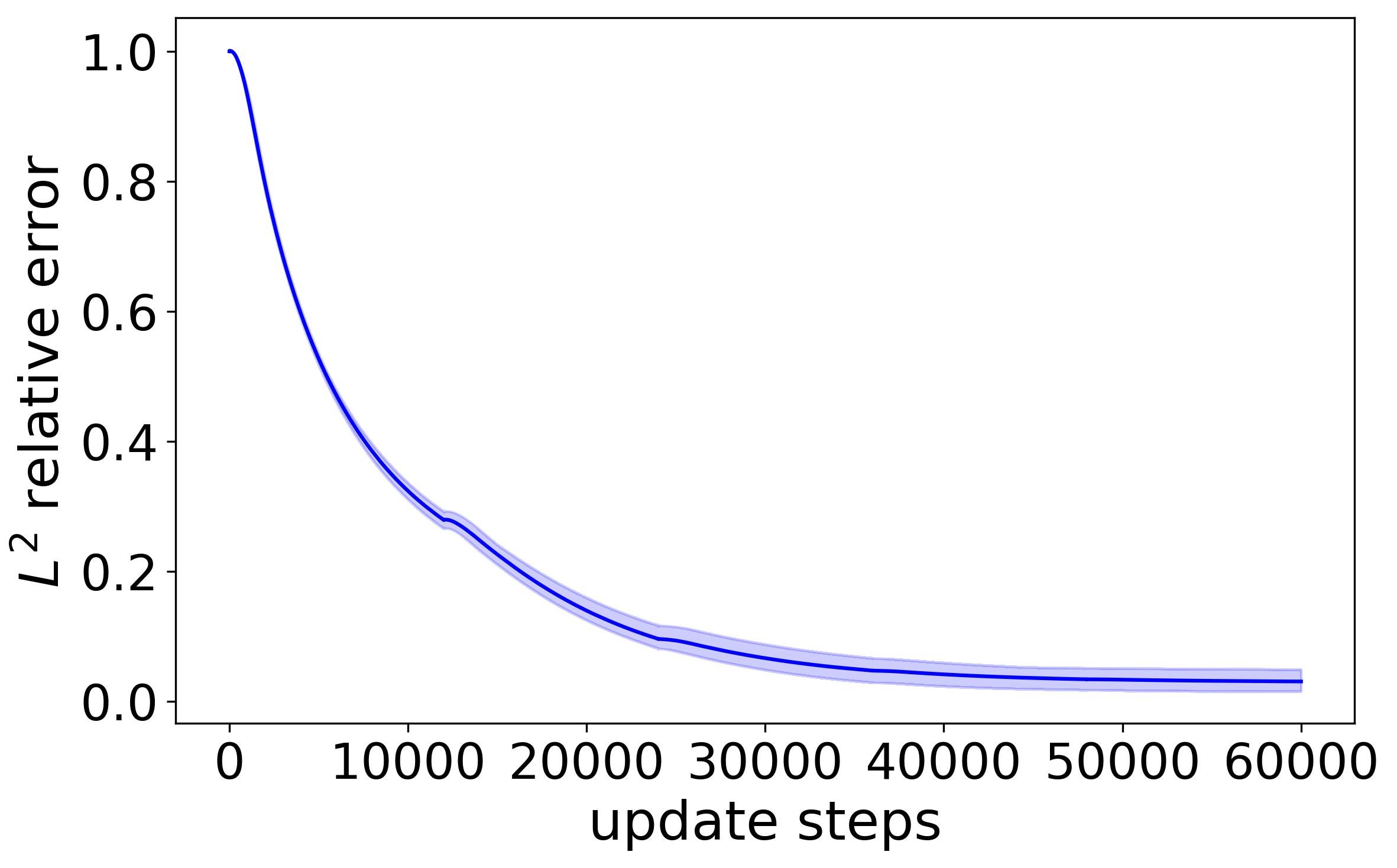}
\caption{The $ L^{2} $ relative error $ \frac{\mathbb{E}|u_{0} -u_{0}^{*}|_{L^{2}(G)}}{\mathbb{E} |u_{0}^{*}|_{L^{2}(G)}} $ over update steps}
\label{fig.exp1d}
\end{subfigure}%
\hfill
\begin{subfigure}[t]{0.47\textwidth}
\centering
\includegraphics[width=\linewidth]{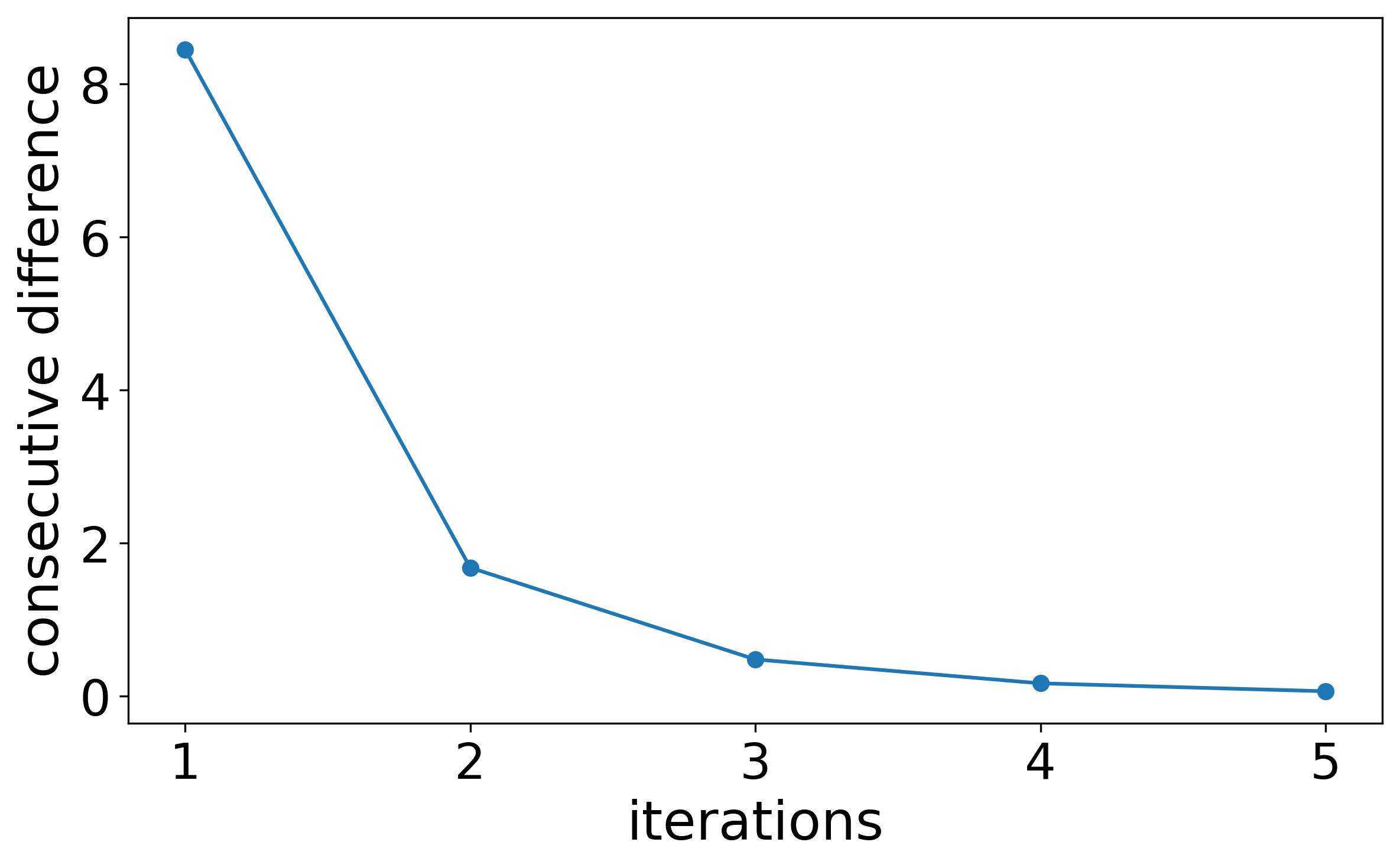}
\caption{The consecutive difference $ \mathbb{E} |u_{\ell}^{n+1}(0)-u_{\ell}^{n}(0)|_{L^{\infty}(G)} $}
\label{fig.exp1e}
\end{subfigure}

\caption{ The numerical results for \cref{ex1}. The computed source function $u_{0}$ is very close to the true source function $u_{0}^{*}$. The relative difference is small, and the $L^{2}$ relative error is 2.1\%. }
\label{fig.exp1}
\end{figure}

\begin{example}
\label{ex1}
Let $ F(u, u_{t}, \nabla u) = \sqrt{1 + u^{2}} + |\nabla u| $, $ a = 10 x y t^{2} $
and the true source function 
\begin{align*}
u^{*}_{0} = \left\{
\begin{aligned}
& 10 e^{\frac{r(x,y)}{r(x,y)-1}}, \quad & r(x,y) \leq 1, \\
& 0, \quad & r(x,y) > 1,
\end{aligned}
\right.
\end{align*}
where \vspace{-3mm}
$$
r(x,y) = 16 \bigg[
\bigg( x-\frac{1}{2} \bigg)^{2} + \frac{1}{2} (y-1)^{2}
\bigg].
$$
The boundary conditions for the forward equation are given by $ f^{*} = u^{*}_{0} \big|_{\Gamma} $.
The lateral Cauchy data, namely $f$ and $g$, are obtained from \cref{eqNoisyData,eqBoundaryCon1,eqBoundaryCon2}.
The numerical result is presented in \cref{fig.exp1}.
\end{example}

From \cref{fig.exp1a,fig.exp1b}, it is evident that the source function $u_{0}^{*}$ is very well recovered, and from \cref{fig.exp1c} one can observe that the relatively large differences occur along the boundary of the initial condition. 
The final $L^{2}$ relative error  is 2.1\%. 
From \cref{fig.exp1d}, it can be seen that the $L^{2}$ relative error remains largely unchanged for different sample paths.
Moreover, from \cref{fig.exp1e} it is seen that, with increasing iterations, the consecutive difference decreases at an exponential rate, which confirms the conclusion of \cref{thmConverge}.

\begin{figure}[!htb]
\centering

\begin{subfigure}[t]{0.3\textwidth}
\centering
\includegraphics[width=\linewidth]{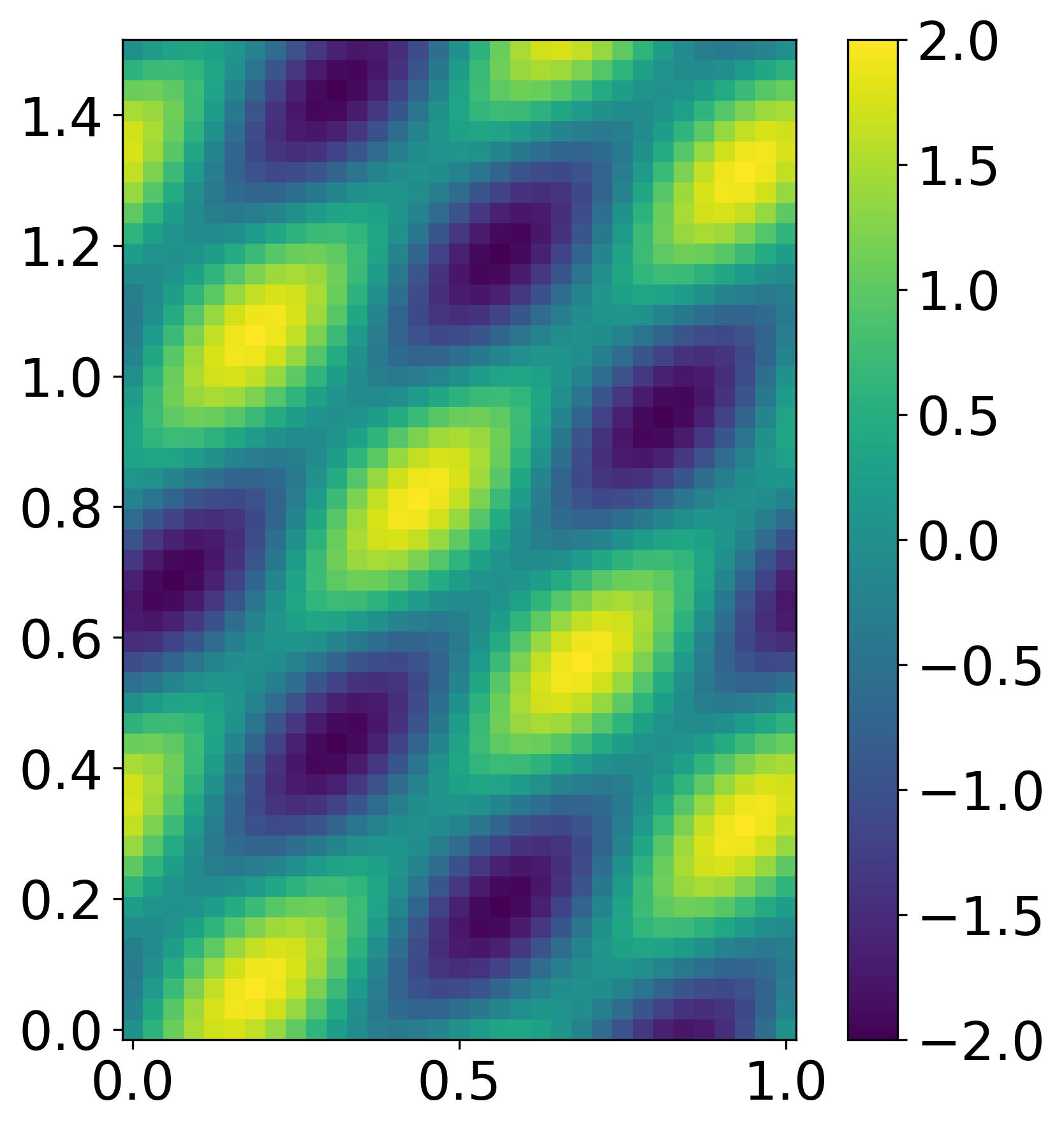}
\caption{The true source function $u_{0}^{*}$.}
\label{fig.exp2a}
\end{subfigure}%
\hfill
\begin{subfigure}[t]{0.3\textwidth}
\centering
\includegraphics[width=\linewidth]{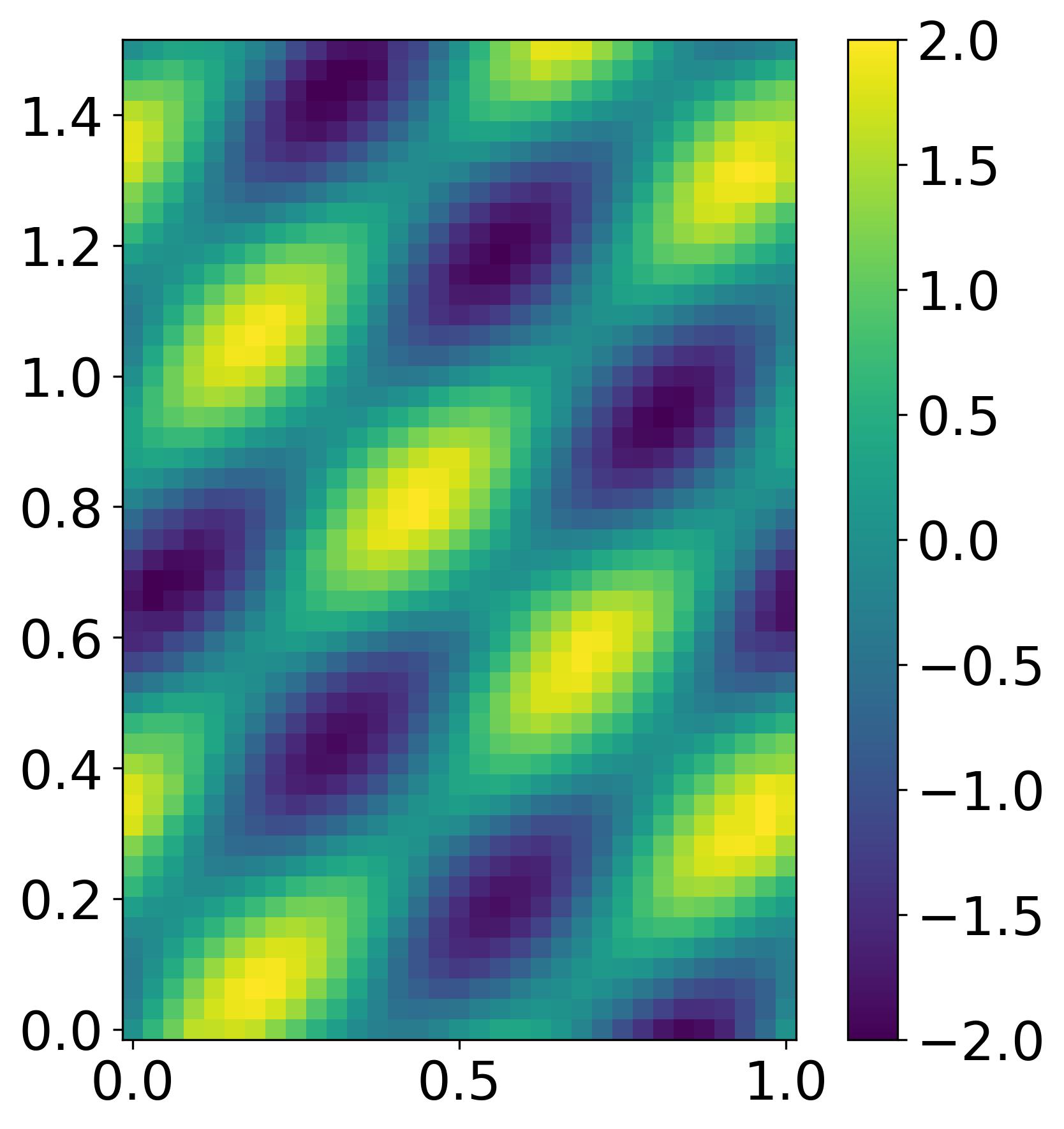}
\caption{The computed source function $u_{0}$.}
\label{fig.exp2b}
\end{subfigure}%
\hfill
\begin{subfigure}[t]{0.3\textwidth}
\centering
\includegraphics[width=1.03 \linewidth]{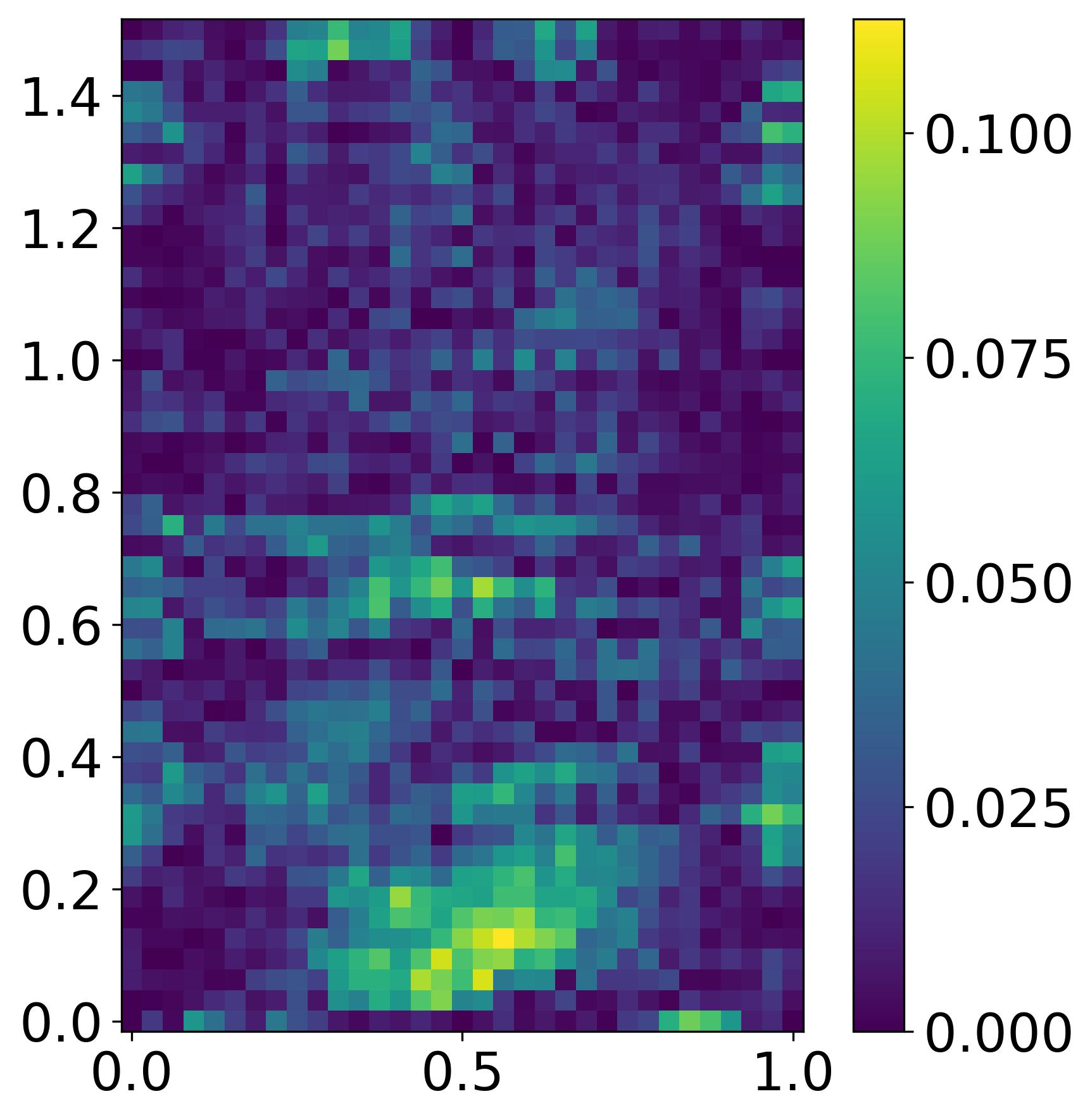}
\caption{The relative difference $ \frac{|u_{0} - u_{0}^{*}|}{|u_{0}^{*}|_{L^{\infty}(G)}} $.}
\label{fig.exp2c}
\end{subfigure}

\vspace{0.5cm} 

\begin{subfigure}[t]{0.47\textwidth}
\centering
\includegraphics[width=\linewidth]{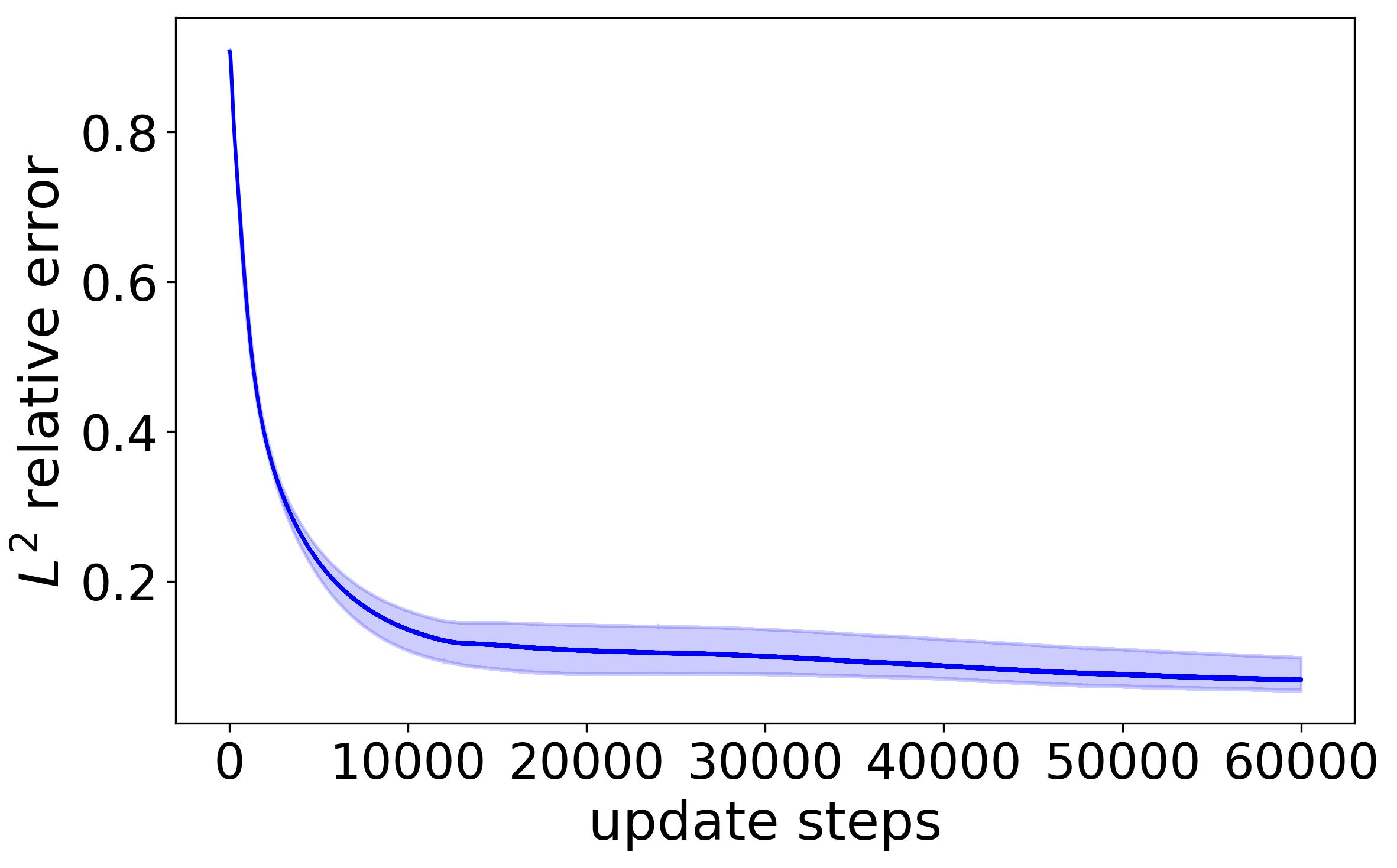}
\caption{The $ L^{2} $ relative error $ \frac{\mathbb{E}|u_{0} -u_{0}^{*}|_{L^{2}(G)}}{\mathbb{E} |u_{0}^{*}|_{L^{2}(G)}} $ over update steps}
\label{fig.exp2d}
\end{subfigure}%
\hfill
\begin{subfigure}[t]{0.47\textwidth}
\centering
\includegraphics[width=\linewidth]{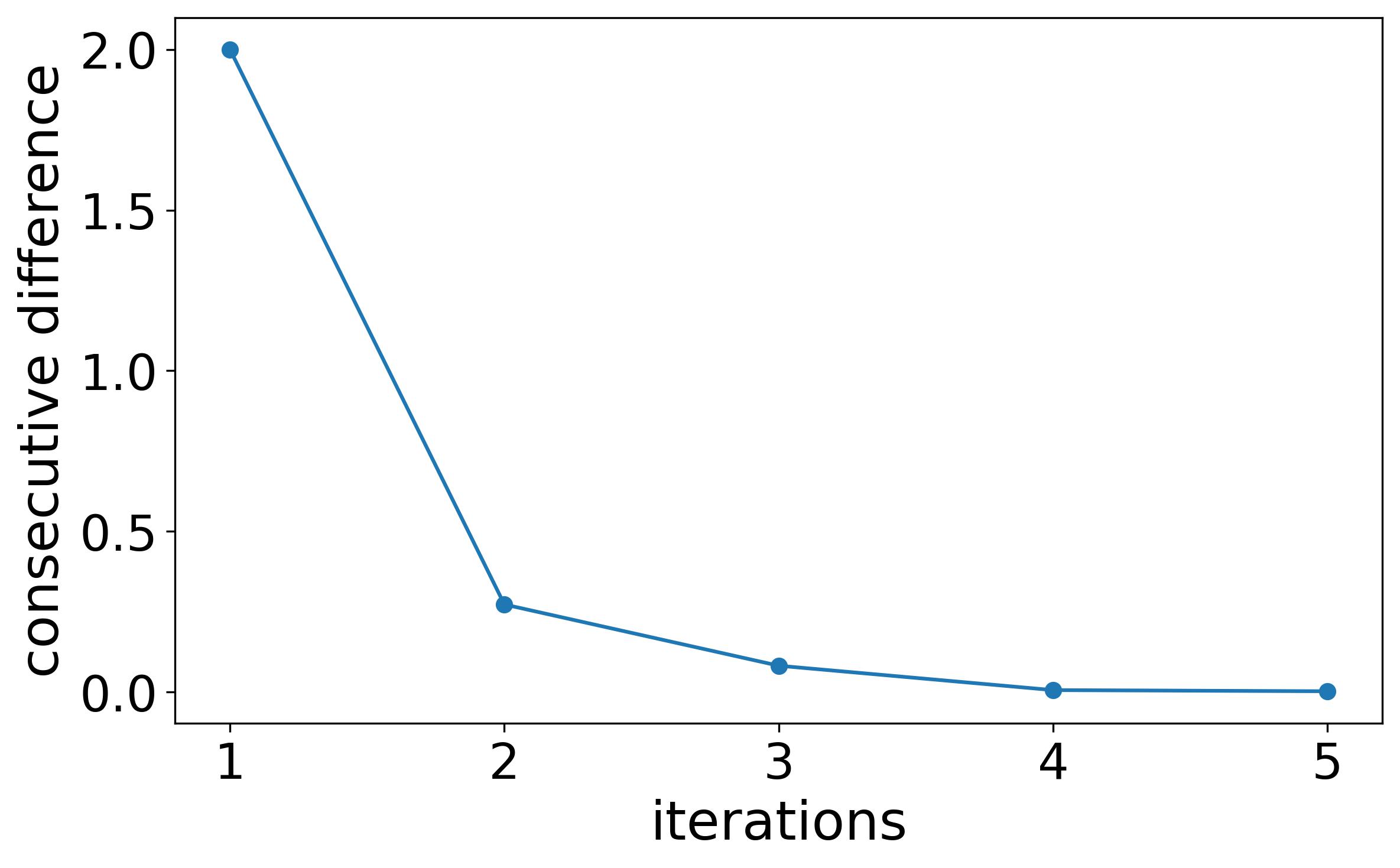}
\caption{The consecutive difference $ \mathbb{E} |u_{\ell}^{n+1}(0)-u_{\ell}^{n}(0)|_{L^{\infty}(G)} $}
\label{fig.exp2e}
\end{subfigure}

\caption{The numerical results for \cref{ex2}. The computed source function $u_{0}$ approximates the true source function $u_{0}^{*}$ well, despite the non-Lipschitz continuity of the nonlinearity $F$. The relative difference is small, and the $L^{2}$ relative error is 6\%. }
\label{fig.exp2}
\end{figure}

\begin{example}
\label{ex2}
Let $ F(u, u_{t}, \nabla u) = \min\{e^{u} + |\nabla u|, 10\} $, $ a = x^{2}+ y^{2} + t^{2} $
and the true source function 
\begin{align*}
u^{*}_{0} =  \sin(2 \pi(x+y)) + \sin(4\pi(x-y))
\end{align*}
The boundary conditions for the forward equation are given by $ f^{*} = u^{*}_{0} \big|_{\Gamma} $.
The lateral Cauchy data  $f$ and $g$  are obtained from \cref{eqNoisyData,eqBoundaryCon1,eqBoundaryCon2}.
The numerical result is presented in \cref{fig.exp2}.

It is worth noting that the nonlinearity $F$ is not Lipschitz continuous, which significantly increases the challenge of the inverse problem.
\cref{fig.exp2a,fig.exp2b} show that the source function $u_{0}^{*}$ is accurately recovered, with a final $L^{2}$ relative error of 6\%. This demonstrates that our method remains robust even when dealing with non-Lipschitz nonlinearities. Moreover, as illustrated in \cref{fig.exp2d}, the $L^{2}$ relative error stays largely consistent across different sample paths.

\end{example}

\begin{figure}[!htb]
\centering

\begin{subfigure}[t]{0.3\textwidth}
\centering
\includegraphics[width=\linewidth]{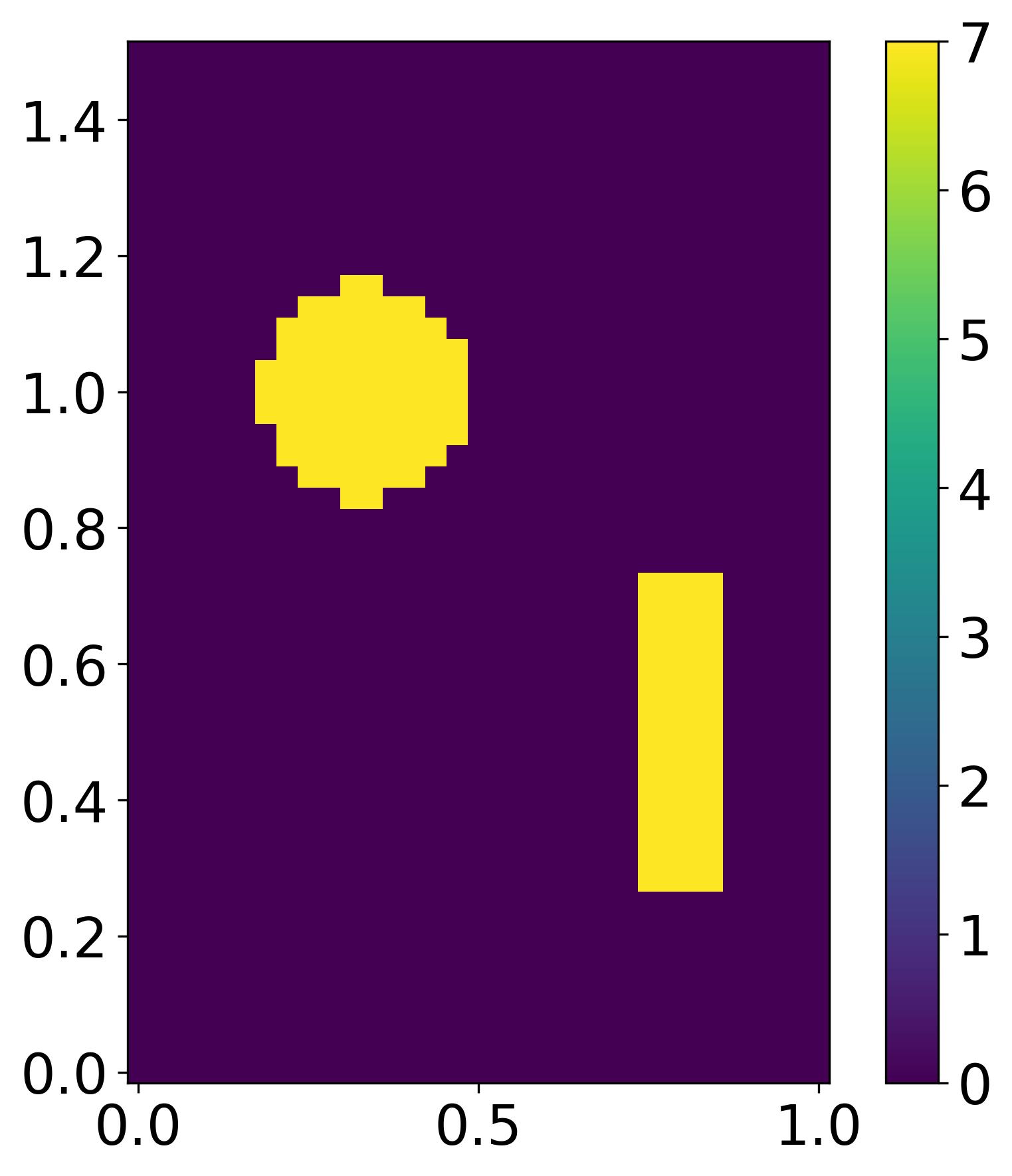}
\caption{The true source function $u_{0}^{*}$.}
\label{fig.exp3a}
\end{subfigure}%
\hfill
\begin{subfigure}[t]{0.3\textwidth}
\centering
\includegraphics[width=\linewidth]{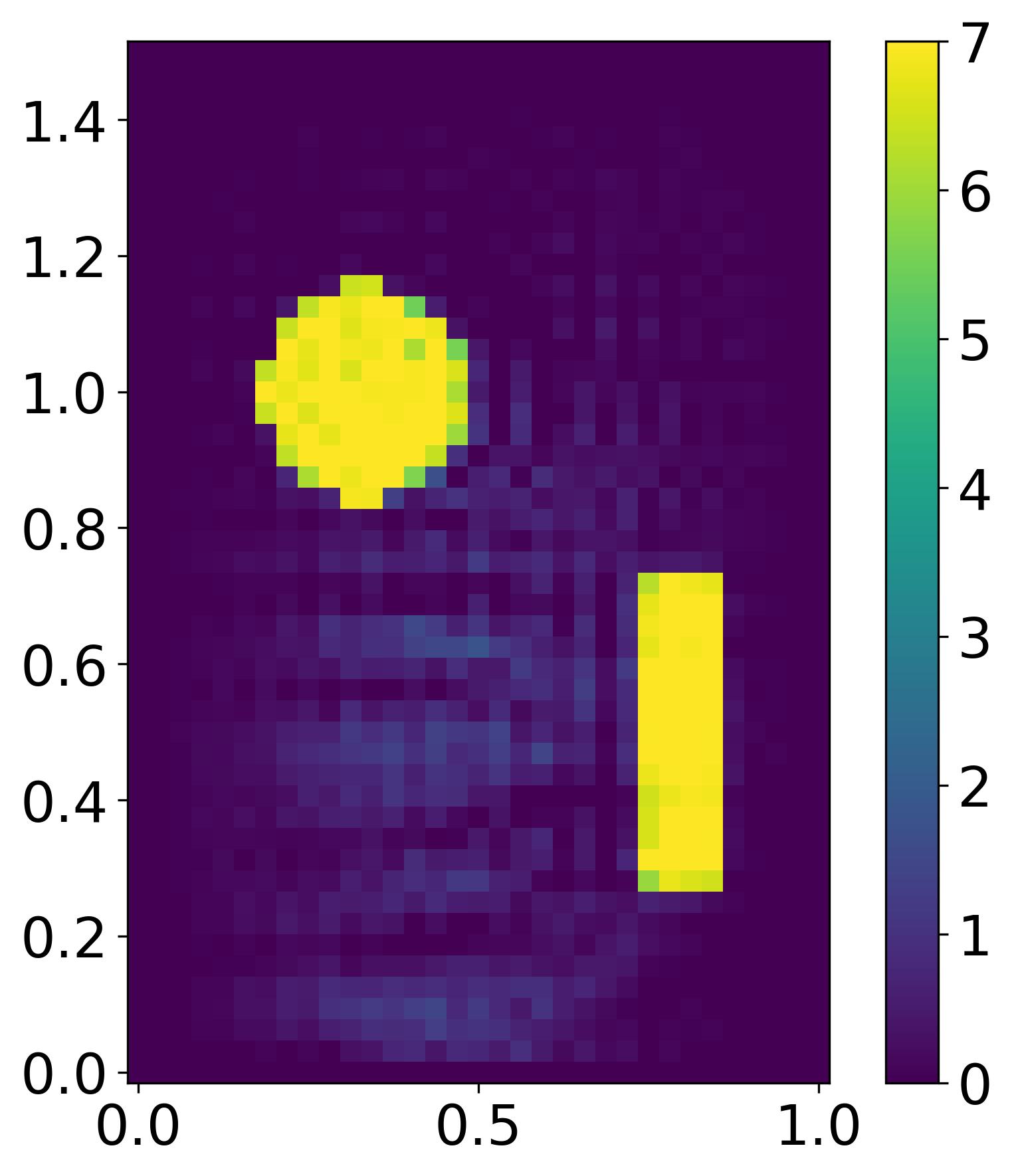}
\caption{The computed source function $u_{0}$.}
\label{fig.exp3b}
\end{subfigure}%
\hfill
\begin{subfigure}[t]{0.3\textwidth}
\centering
\includegraphics[width=1.07 \linewidth]{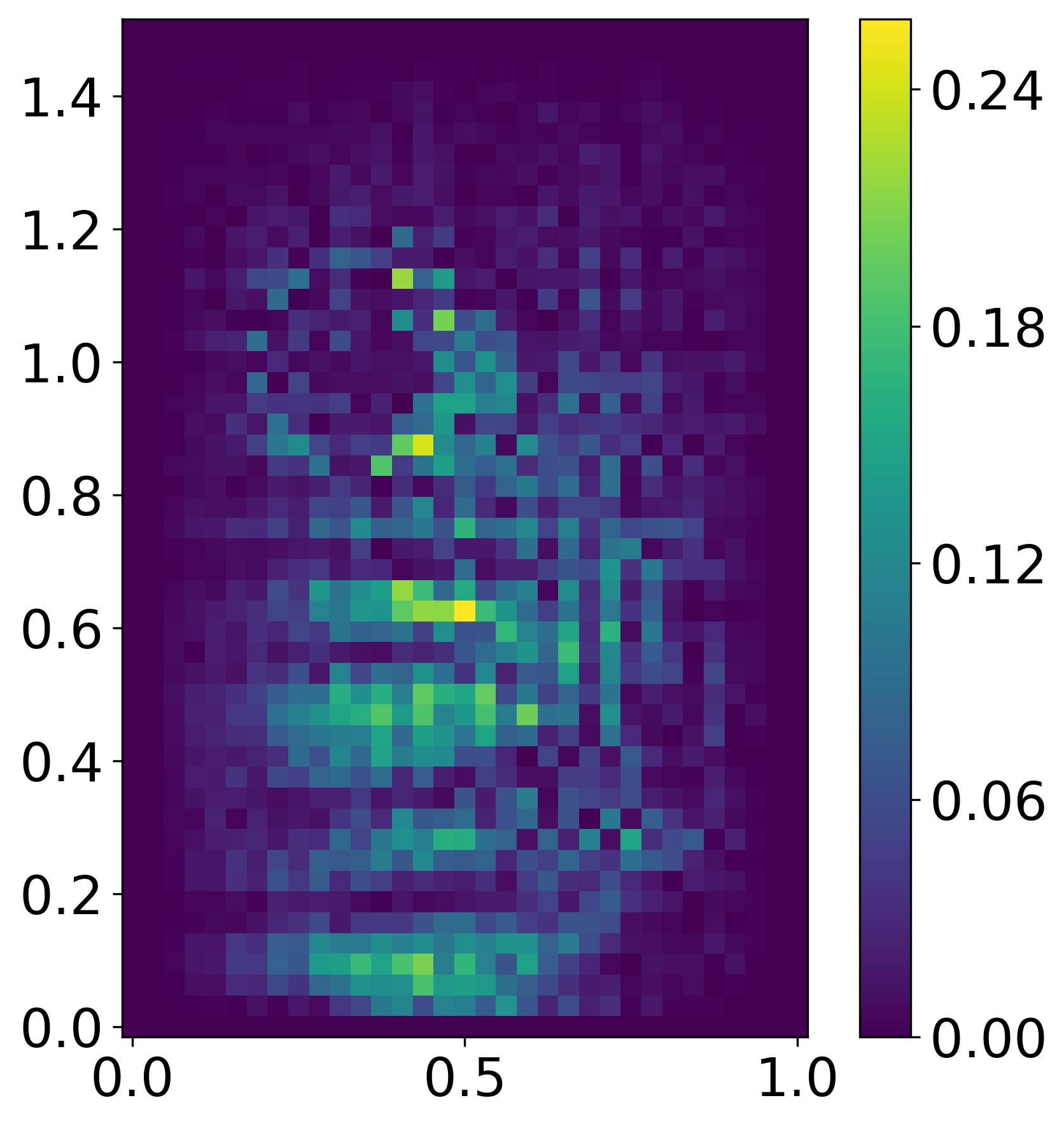}
\caption{The relative difference $ \frac{|u_{0} - u_{0}^{*}|}{|u_{0}^{*}|_{L^{\infty}(G)}} $.}
\label{fig.exp3c}
\end{subfigure}

\vspace{0.5cm} 

\begin{subfigure}[t]{0.47\textwidth}
\centering
\includegraphics[width=\linewidth]{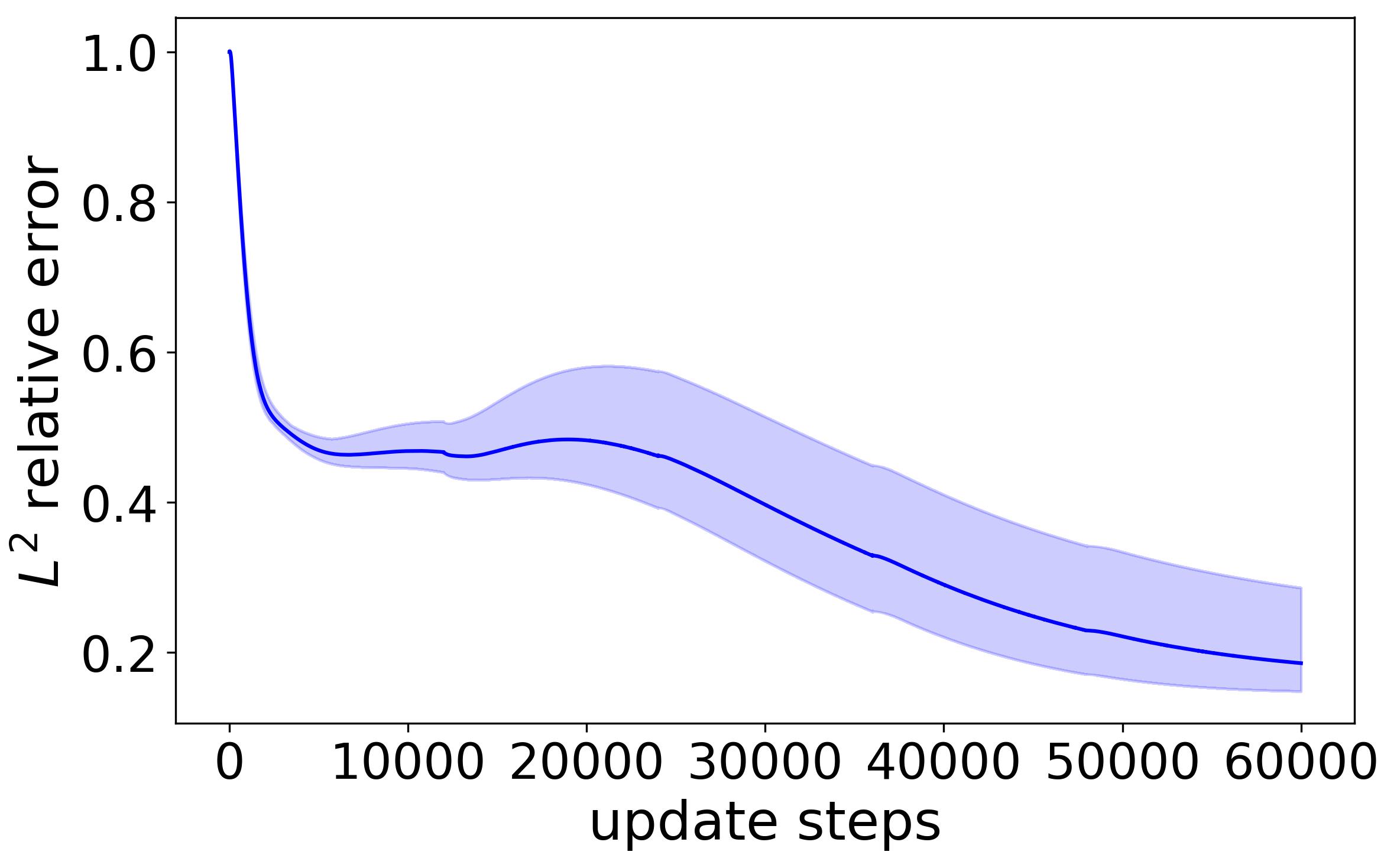}
\caption{The $ L^{2} $ relative error $ \frac{\mathbb{E}|u_{0} -u_{0}^{*}|_{L^{2}(G)}}{\mathbb{E} |u_{0}^{*}|_{L^{2}(G)}} $ over update steps}
\label{fig.exp3d}
\end{subfigure}%
\hfill
\begin{subfigure}[t]{0.47\textwidth}
\centering
\includegraphics[width=\linewidth]{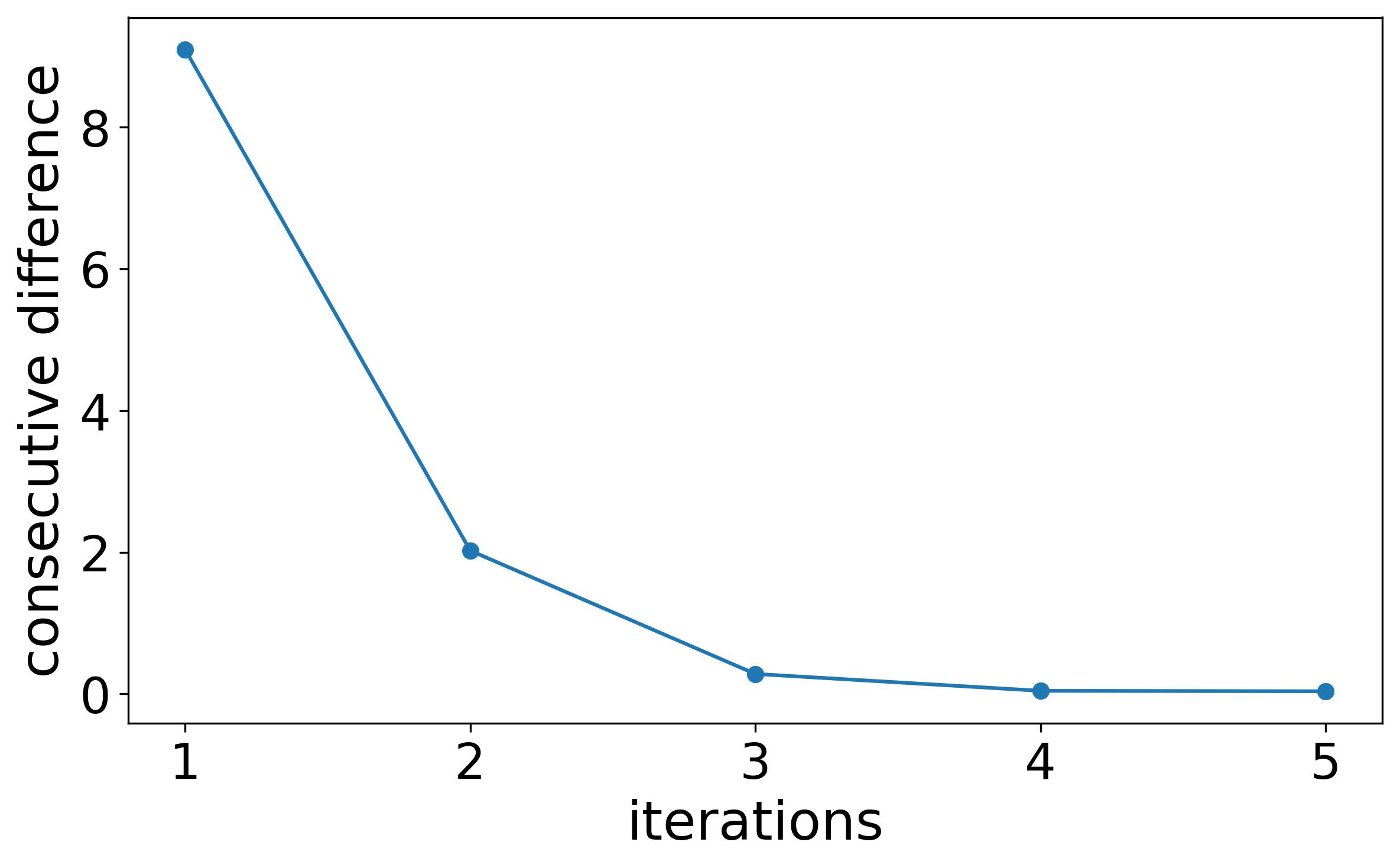}
\caption{The consecutive difference $ \mathbb{E} |u_{\ell}^{n+1}(0)-u_{\ell}^{n}(0)|_{L^{\infty}(G)} $}
\label{fig.exp3e}
\end{subfigure}

\caption{The numerical results for \cref{ex3}. The computed source function $u_{0}$ closely approximates the true source function $u_{0}^{*}$. Despite both $f$ and $g$ being stochastic processes, which further complicates the problem, the relative difference remains minimal and the $L^{2}$ relative error is 17.9\%.}  \label{fig.exp3}
\end{figure}

\begin{example}
\label{ex3}
Let $ F(u, u_{t}, \nabla u) = \sqrt{1 + u^{2}} + |\nabla u| $, $ a = 10 x y t^{2} $
and 
\begin{align*}
\widetilde{u}^{*}_{0} = \left\{
\begin{aligned}
    & 7, \quad & &\bigg( x- \frac{7}{3} \bigg)^{2} + (y - 4)^{2} < \frac{1}{40} \text{ or } 
    \max \bigg\{ 
        16 \bigg |x-\frac{14}{5}\bigg|, 4 \bigg| y - \frac{7}{2} \bigg|
    \bigg\} <1
    \\
    & 0, \quad && \text{ others}.
\end{aligned}
\right.
\end{align*}
Letting $ \widetilde{G} = (0,5) \times (0,7.5) $, $ \widetilde{N}_{x} = 60 $ and $ \widetilde{N}_{y} = 240 $, we solve \cref{eqDisFor} in $ \widetilde{G} $.
The boundary conditions for the forward equation are given by $ \widetilde{f}^{*} = \widetilde{u}^{*}_{0} \big|_{\partial \widetilde{G}} $.
The true source function is given by $\widetilde{u}_{0}^{*}$  restricted to  $[0,1]\times[2,3]\times[3,4.5]$ .
The lateral Cauchy data  $f$ and $g$  are also obtained from  the restricted true solution function $\widetilde{u}^{*}$.
In this case, both $f$ and $g$ are stochastic processes as well, which further complicates the problem.
The numerical result is presented in \cref{fig.exp3}.

\cref{fig.exp3a,fig.exp3b} show that the source function $u_{0}^{*}$ is accurately recovered, with a final $L^{2}$ relative error of 17.9\%. 
Since both $f$ and $g$ are stochastic processes in this case, additional randomness is introduced into the problem. 
As can be seen from \cref{fig.exp3d}, although the $L^{2}$ relative errors vary  for different sample paths, the averaged error remains  stable.

\end{example}

\begin{figure}[!htb]
\centering

\begin{subfigure}[t]{0.3\textwidth}
\centering
\includegraphics[width=\linewidth]{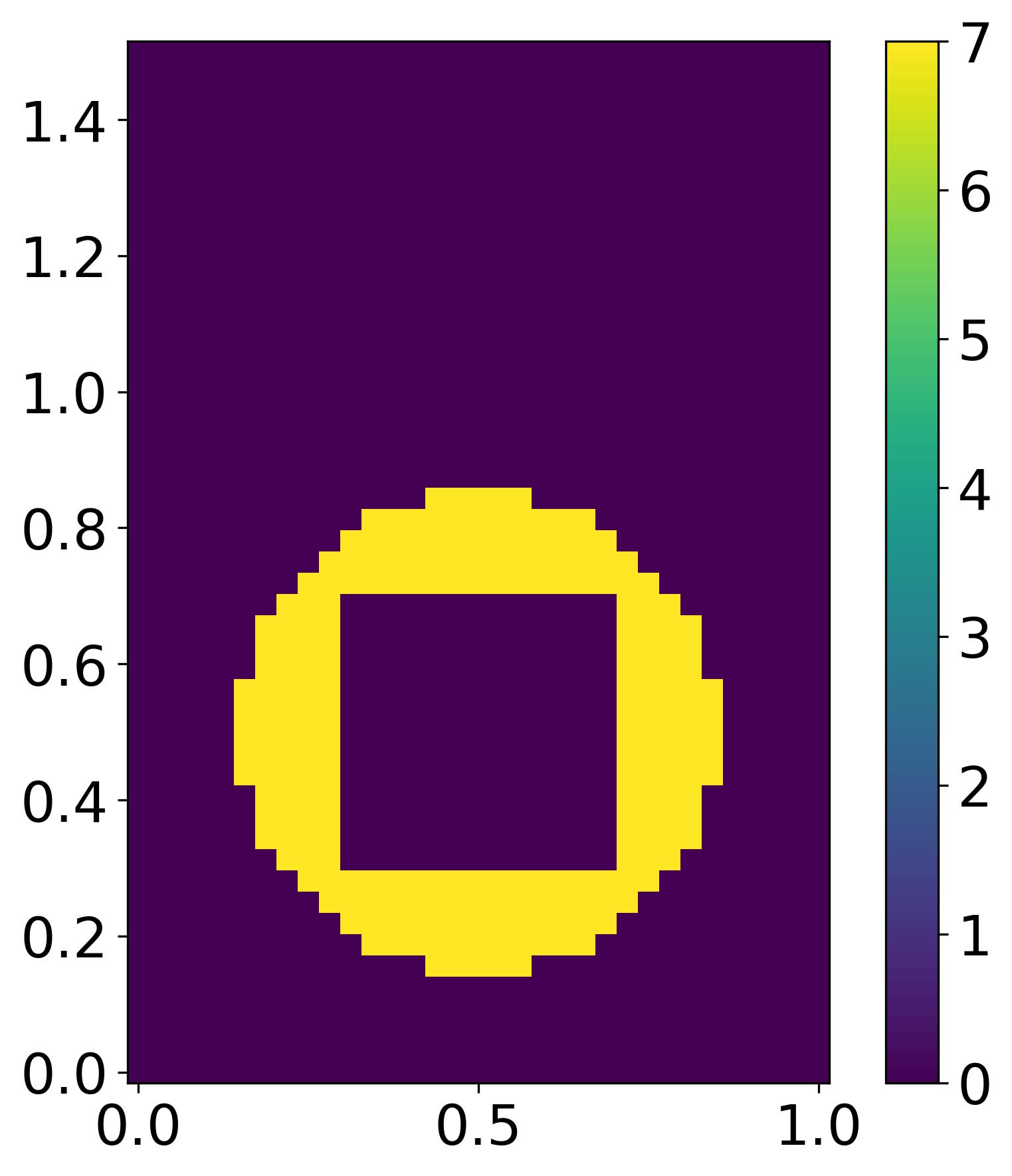}
\caption{The true source function $u_{0}^{*}$.}
\label{fig.exp4a}
\end{subfigure}%
\hfill
\begin{subfigure}[t]{0.3\textwidth}
\centering
\includegraphics[width=\linewidth]{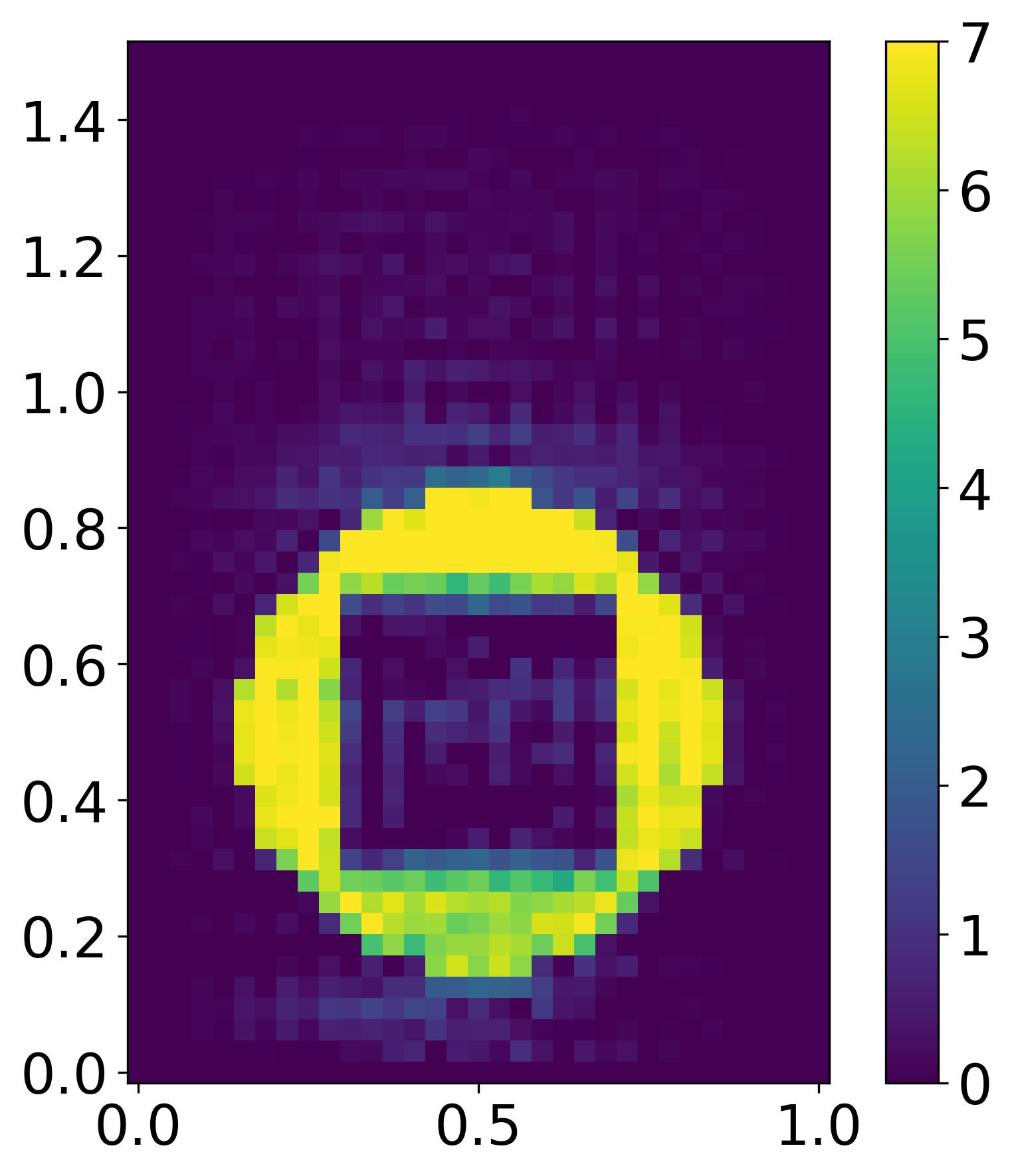}
\caption{The computed source function $u_{0}$.}
\label{fig.exp4b}
\end{subfigure}%
\hfill
\begin{subfigure}[t]{0.3\textwidth}
\centering
\includegraphics[width=1.045 \linewidth]{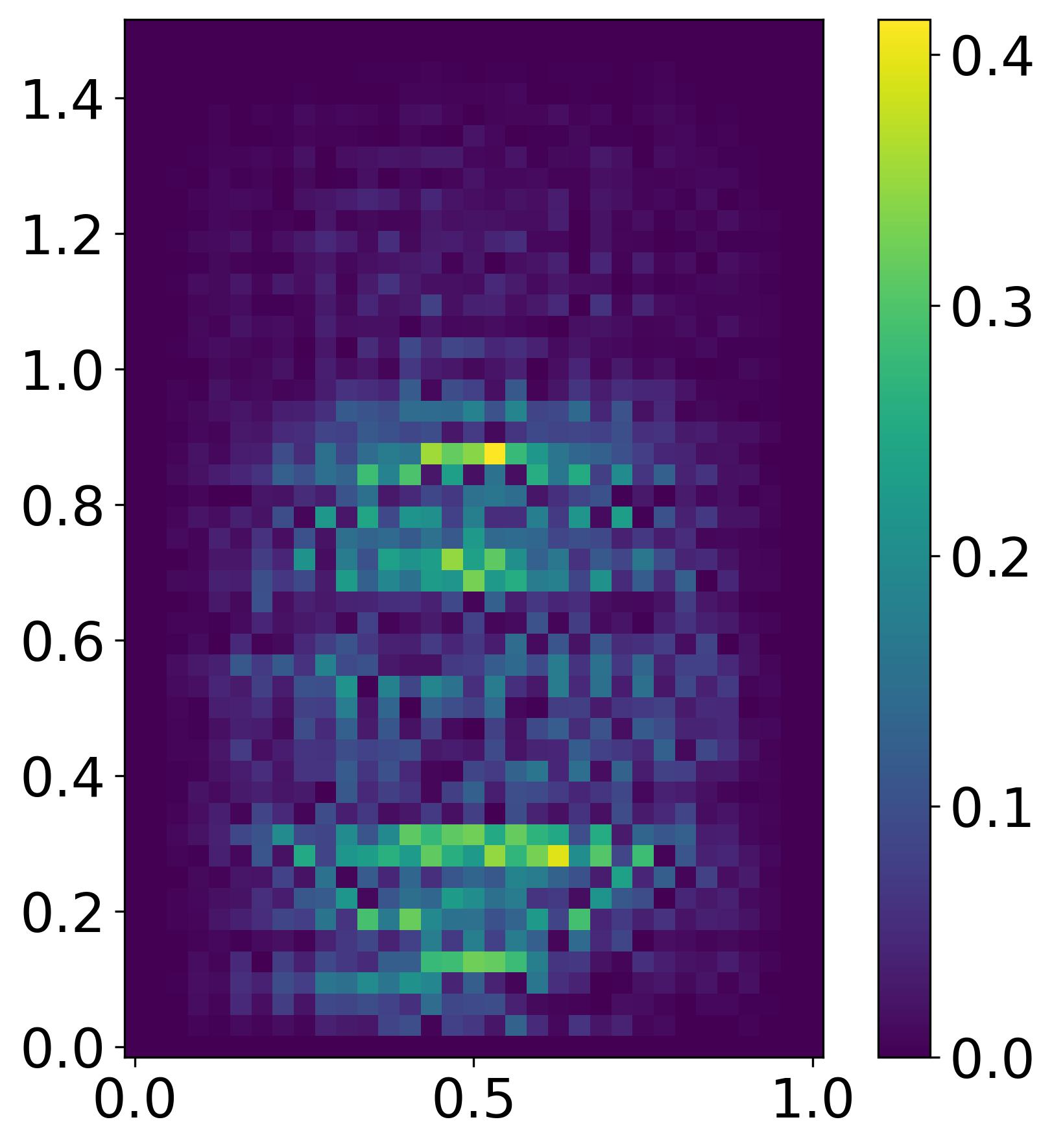}
\caption{The relative difference $ \frac{|u_{0} - u_{0}^{*}|}{|u_{0}^{*}|_{L^{\infty}(G)}} $.}
\label{fig.exp4c}
\end{subfigure}

\vspace{0.5cm} 

\begin{subfigure}[t]{0.47\textwidth}
\centering
\includegraphics[width=\linewidth]{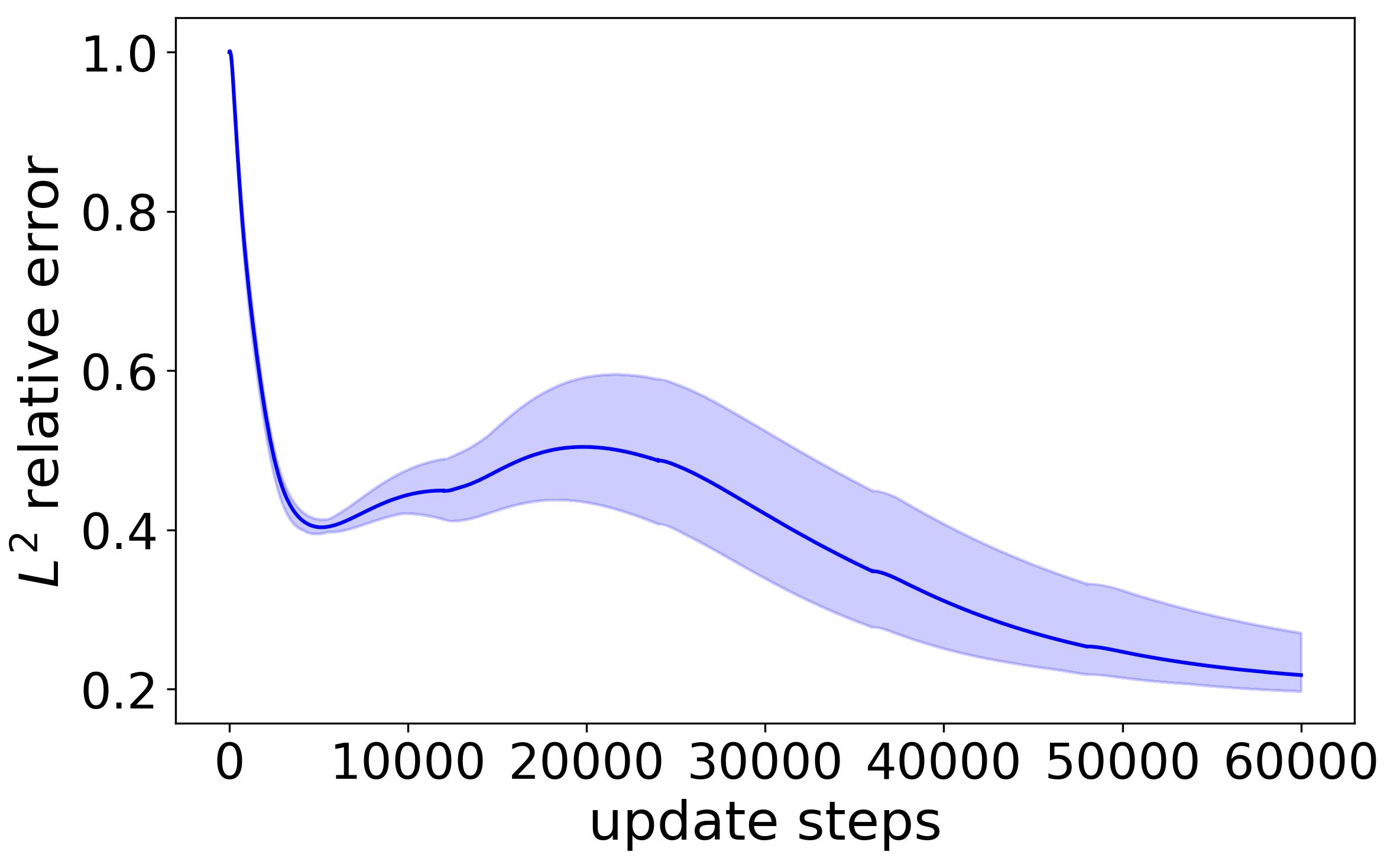}
\caption{The $ L^{2} $ relative error $ \frac{\mathbb{E}|u_{0} -u_{0}^{*}|_{L^{2}(G)}}{\mathbb{E} |u_{0}^{*}|_{L^{2}(G)}} $ over update steps}
\label{fig.exp4d}
\end{subfigure}%
\hfill
\begin{subfigure}[t]{0.47\textwidth}
\centering
\includegraphics[width=\linewidth]{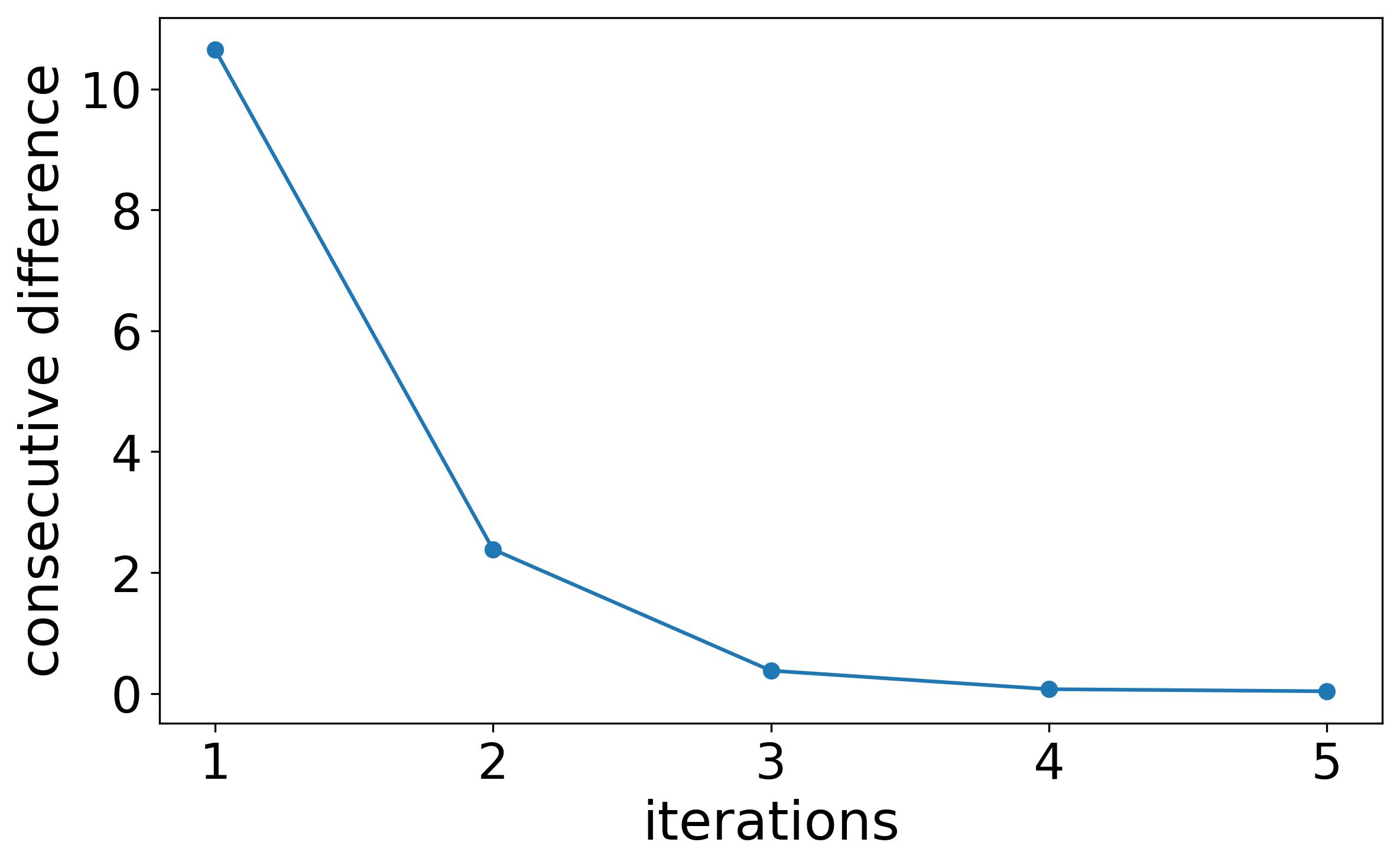}
\caption{The consecutive difference $ \mathbb{E} |u_{\ell}^{n+1}(0)-u_{\ell}^{n}(0)|_{L^{\infty}(G)} $}
\label{fig.exp4e}
\end{subfigure}

\caption{The numerical results for \cref{ex4}. The computed source function $u_{0}$ closely approximates the true source function $u_{0}^{*}$. In this case, both $f$ and $g$ are stochastic processes, which further complicates the problem.
The relative difference remains minimal and the $L^{2}$ relative error is 20.7\%.}  \label{fig.exp4}
\end{figure}

\begin{example}
\label{ex4}
Let $ F(u, u_{t}, \nabla u) = \sqrt{1 + u^{2}} + |\nabla u| $, $ a = 10 x y t^{2} $
and 
\begin{align*}
\widetilde{u}^{*}_{0} = \left\{
\begin{aligned}
    & 7, \quad & &\bigg( x- \frac{5}{2} \bigg)^{2} +  \bigg(y - \frac{7}{2} \bigg)^{2} < \frac{1}{8} \text{ or } 
    \max \bigg\{ 
          \bigg |x-\frac{5}{2}\bigg|,   \bigg| y - \frac{7}{2} \bigg|
    \bigg\} > \frac{1}{5}
    \\
    & 0, \quad && \text{ others}.
\end{aligned}
\right.
\end{align*}
Letting $ \widetilde{G} = (0,5) \times (0,7.5) $, $ \widetilde{N}_{x} = 60 $ and $ \widetilde{N}_{y} = 240 $, we solve \cref{eqDisFor} in $ \widetilde{G} $.
The boundary conditions for the forward equation are given by $ \widetilde{f}^{*} = \widetilde{u}^{*}_{0} \big|_{\partial \widetilde{G}} $.
The true source function is given by $\widetilde{u}_{0}^{*}$  restricted to $[0,1]\times[2,3]\times[3,4.5]$.
The lateral Cauchy data  $f$ and $g$  are also obtained from  the restricted true solution function $\widetilde{u}^{*}$.
In this case, both $f$ and $g$ are stochastic processes as well, which further complicate the problem.
The numerical result is presented in \cref{fig.exp4}.

\cref{fig.exp4a,fig.exp4b} show that the source function $u_{0}^{*}$ is accurately recovered, with a final $L^{2}$ relative error of 20.7\%.
Although the stochastic nature of $f$ and $g$ introduces additional randomness, and the absence of normal derivative information on the lower boundary (with final time $T=1$) makes the reconstruction in the lower portion less accurate than in the upper part, the overall recovery remains robust. 
Moreover, as seen in \cref{fig.exp4c}, the regions with relatively larger error are occurred to the boundary of the source function, while errors elsewhere are very small. 
Finally, \cref{fig.exp4d} indicates that although the $L^{2}$ relative errors vary across different sample paths, the averaged error remains stable.

\end{example}

\section{Concluding remarks}

In this paper, we have proposed a new approach for solving inverse source problems for semilinear stochastic hyperbolic equations. 
Our method combines fixed-point iteration and Carleman estimates, enabling global convergence without the need for a good initial guess. 
To address the challenges arose by the stochastic nature of the equations, such as the non-differentiability of Brownian motion sample paths and the lack of compactness in the solution space, we redesigned the Tikhonov functional and developed a new Carleman estimate suitable for stochastic hyperbolic equations.

Theoretical analysis established the convergence and stability of the proposed method, with  \cref{thmConverge} confirming H\"{o}lder-type stability with respect to data noise.
Numerical experiments demonstrated that our method is robust, efficient, and capable of reconstructing the source function using only partial boundary observations, even in the presence of strong nonlinearity and randomness.
In contrast, as shown in \cref{fig:compareSGD_Adam}, the classical conjugate gradient method may not handle randomness well, whereas our approach remains computationally efficient.

Future work may include extending the approach to more general classes of stochastic partial differential equations and improving computational efficiency.

\appendix

\end{document}